\documentclass[10pt,oneside]{article}

\usepackage[utf8]{inputenc} %
\usepackage[T1]{fontenc}    %
\usepackage{url}            %
\usepackage{booktabs}       %
\usepackage{amsfonts}       %
\usepackage{nicefrac}       %
\usepackage{microtype}      %

\usepackage{lmodern}
\usepackage{times}

\usepackage{amssymb,amsmath,amsthm}
\usepackage{bbold,stmaryrd}
\usepackage[short]{optidef}

\usepackage{framed}

\usepackage{xcolor}
\usepackage{subfigure,graphicx,epsfig,tikz,caption}
\usepackage[normalem]{ulem}
\usepackage{enumerate}
\usepackage{verbatim}
\usepackage{makeidx,latexsym}
\usepackage[colorlinks=true]{hyperref}%

\usepackage{bm}
\usepackage{stmaryrd}

\usepackage[numbers,sort&compress,square,comma]{natbib}

\usepackage{algorithmic,algorithm}

\usepackage[english]{babel}
\usepackage{bbm}
\usepackage{pgfplots}

    \usepackage{parskip}
\usepackage[letterpaper,margin=1.1in]{geometry}

    \usepackage{thmtools,thm-restate}

    \declaretheorem{corollary}
    \declaretheorem{lemma}
    \declaretheorem{proposition}

    \declaretheoremstyle[qed=$\square$]{definitionwithend}
    \declaretheorem[style=definitionwithend]{definition}
    
    \declaretheorem[style=definitionwithend]{example}
    \declaretheorem[style=definitionwithend]{remark}

\usepackage{url}
\usepackage[colorlinks=true]{hyperref}
\usepackage[numbers,sort&compress,square,comma]{natbib}

\definecolor{gold}{rgb}{0.85,0.65,0}

\renewcommand{\theequation}{\mbox{\arabic{equation}}}
\newcommand{\be}{\begin{eqnarray}}
\newcommand{\ee}[1]{\label{#1}\end{eqnarray}}
\newcommand{\ese}{\end{eqnarray*}}
\newcommand{\bse}{\begin{eqnarray*}}

\def\fnote#1{\footnote}

\def\P{{\mathbb{P}}}

\def\R{{\mathbb{R}}}

\def\Z{{\mathbb{Z}}}

\def\cM{{\cal M}}

\def\cP{{\cal P}}
\def\cQ{{\cal Q}}
\def\cR{{\cal R}}

\def\cV{{\cal V}}

\def\cX{{\cal X}}

\newcommand{\vc}{\bm{c}}

\newcommand{\va}{\bm{a}}
\newcommand{\vA}{\bm{A}}

\newcommand{\vh}{\bm{h}}
\newcommand{\vr}{\bm{r}}
\newcommand{\vq}{\bm{q}}

\newcommand{\vW}{\bm{W}}
\newcommand{\vz}{\bm{z}} %
\newcommand{\vx}{\bm{x}} %
\newcommand{\vy}{\bm{y}} %
\newcommand{\vpi}{\bm{\pi}} %
\newcommand{\vb}{\bm{b}} %

\newcommand{\vbeta}{\bm{\beta}}

\newcommand{\vxi}{\bm{\xi}}

\newcommand{\valpha}{\bm{\alpha}}

\newcommand{\vell}{\bm{\ell}}

\DeclareMathOperator{\conv}{conv}

\def\ext{{\mathop{\rm ext}}}

\def\log{\mathop{{\rm log}}}

\newcommand{\st}{\text{s.t.}}

\newcommand {\beqn}{\begin{equation}}\newcommand {\eeqn}{\end{equation}}
\newcommand {\beqan}{\begin{eqnarray}}\newcommand {\eeqan}{\end{eqnarray}}
\newcommand {\beqa}{\begin{eqnarray*}}\newcommand {\eeqa}{\end{eqnarray*}}
\newcommand{\skipit}  [1] {}

\newtheorem{theorem}{Theorem}[section]

\begin{document}
\title{Joint Chance-Constrained Programs and the Intersection of Mixing Sets through a Submodularity Lens}
\author{
	Fatma K{\i}l{\i}n\c{c}-Karzan\thanks{Tepper School of Business, Carnegie Mellon University, Pittsburgh, PA 15213, USA, \url{fkilinc@andrew.cmu.edu}}
	\and 
	Simge K\"{u}\c{c}\"{u}kyavuz\thanks{Department of Industrial Engineering and Management Sciences, Northwestern University, Evanston, IL 60208, USA, \url{simge@northwestern.edu}}
	\and
	Dabeen Lee\thanks{Discrete Mathematics Group, Institute for Basic Science (IBS), Daejeon 34126, Republic of Korea, \url{dabeenl@ibs.re.kr}}
}	

\date{\today}

\maketitle

\begin{abstract}
A particularly important substructure in modeling joint linear chance-constrained programs with random right-hand sides and finite sample space is the intersection of mixing sets with common binary variables (and possibly a knapsack constraint). 
In this paper, we first revisit basic mixing sets by establishing a strong and previously unrecognized connection to submodularity. In particular, we show that mixing inequalities with binary variables are nothing but the polymatroid inequalities associated with a specific submodular function. This submodularity viewpoint enables us to unify and extend existing results on valid inequalities and convex hulls of the intersection of multiple mixing sets with common binary variables. Then, we study such intersections under an additional linking constraint lower bounding a linear function of the continuous variables. This is motivated from  the desire to exploit the information encoded in the knapsack constraint arising in joint linear CCPs via the quantile cuts. We propose a new class of valid inequalities and characterize when this new class  along with the mixing inequalities are sufficient to describe the convex hull. 
\end{abstract}

\section{Introduction}\label{sec:intro}

Given integers $n,k\geq 1$, a matrix $\vW=\left\{w_{i,j}\right\}\in\R^{n\times k}_+$, 
a vector $\vell\in\R_+^k$ and a nonnegative number $\varepsilon\geq 0$, we consider the mixed-integer set defined by
\begin{subequations}\label{eq:general-set}
	\begin{align}
	&y_j +w_{i,j} z_i \geq w_{i,j},\quad\qquad \qquad\forall i\in[n],~\forall j\in[k],\label{bigM}\\
	&y_j\geq \ell_j,~~~\qquad\qquad\qquad\qquad\forall j\in[k],\label{lb}\\
	&y_1+\cdots+y_k\geq\varepsilon+\sum_{j\in[k]}\ell_j,\label{linking}\\
	&\vy\in\R^k,~\vz\in\{0,1\}^n.\label{binary}
	\end{align}
\end{subequations}
We denote this set by $\cM(\vW,\vell,\varepsilon)$. When $\vW\in\R^{n\times k}_+ $, constraints~\eqref{bigM} are often called big-$M$ constraints, and constraints~\eqref{lb} impose lower bounds on the continuous variables $\vy$. Notice that~\eqref{linking} is a constraint linking all continuous variables, but it is non-redundant only if $\varepsilon$ is strictly positive. We will refer to~\eqref{linking} as the {\it linking constraint}. When $k=1$, $\vell=\bm{0}$, and $\varepsilon=0$, the set $\cM(\vW,\vell,\varepsilon)$ is nothing but %
what is commonly referred to as the mixing set (with binary variables) %
in the literature~\cite{abdi2016mixing-knapsack,kucukyavuz2012mixing,liu2018intersection,luedtke2010integer,zhao2017joint-knapsack}. Sets of the form $\cM(\vW,\bm{0},0)$ for general $k>1$ were first considered by \citet{atamturk2000mixed}; we will call the set $\cM(\vW,\bm{0},0)$ a {\it joint mixing set} in order to emphasize that $k$ can be taken to be strictly greater than 1. We will refer to a set of the form $\cM(\vW,\vell,\varepsilon)$ for general $\vell,\varepsilon$ as a {\it joint mixing set with lower bounds}. Finally, we will refer to a set of the form $\cM(\vW,\mathbf 0, \varepsilon)$ for general $\varepsilon$ as a {\it joint mixing set with a linking constraint}.

Our motivation for studying the structure of $\cM(\vW,\vell,\varepsilon)$ comes from joint linear chance-constrained programs (CCPs) with right-hand side uncertainty: given a probability space $(\bm{\Xi}, \mathcal{F},\P)$, %
a joint linear CCP with right-hand side uncertainty is an optimization problem of the following form:
\begin{subequations} \label{eq:ccp}
\begin{align}
\min\quad& \vh^\top \vx\label{ccp-obj}\\
\st\quad& \P\left[\vA \vx\geq \vb(\vxi)\right]\geq 1-\epsilon\label{ccp-cons}\\
&\vx\in \cX\subseteq \R^m, \label{ccp-vars}
\end{align}
\end{subequations}
where $\cX\subseteq \R^m$ is a domain for the decision variables $\vx$, $\epsilon\in(0,1)$ is a risk level, $\vb(\vxi)\in\R^k$ is the random right-hand side vector that depends on the random variable $\vxi\in \bm{\Xi}$, and $\vA,\vh$ are matrices of appropriate dimension. For $k=1$ (resp., $k>1$), inequality \eqref{ccp-cons} is referred to as an individual (resp., joint) chance constraint. Here, we seek to find a solution $\vx\in\cX$ satisfying the chance constraint~\eqref{ccp-cons}, enforcing that $\vA \vx\geq \vb(\vxi)$ holds with probability at least the given confidence level $1-\epsilon$, while minimizing the objective~\eqref{ccp-obj}. Joint chance constraints are used to model risk-averse decision-making problems in various applications, such as supply chain logistics \cite{LR07,Minjiao,Murr2000,LK18}, chemical processes \cite{Henrion2001,Henrion2003}, water quality management \cite{TL1999}, and energy \cite{WGW12} (see~\cite{prekopa2003ccp} for further background and an extensive list of applications). 

Problems with joint chance constraints are known to be notoriously challenging because the resulting feasible region is nonconvex even if all other constraints $\vx\in\cX$ and the restrictions inside the chance constraints are convex. %
Moreover, the sample space $\bm{\Xi}$ is typically continuous in practice, and the probability distribution $\mathbb{P}$ quantifying the uncertainty is often unavailable to the decision-maker. Consequently, the classical solution method is to use the Sample Average Approximation (SAA). The basic idea of SAA is to approximate $\bm{\Xi}$ via a set of sample scenarios $\vxi^1,\ldots, \vxi^n$ and reduce the problem to the case with a finite-sample distribution; we refer the interested reader to \cite{CC05,CC06,luedtke2008saa} for further details of SAA for CCPs. 

Inspired by this, joint linear CCPs with the finite sample space assumption have been extensively studied in the literature~\cite{abdi2016mixing-knapsack,kucukyavuz2012mixing,liu2018intersection,luedtke2010integer,zhao2017joint-knapsack}. %
That is to assume that  $\bm{\Xi}=\left\{\vxi^1,\ldots,\vxi^n\right\}$ for some integer $n\geq 1$ and that $\P\left[\vxi=\vxi^i\right]=p_i$ for $i\in[n]$ for some $p_1,\ldots,p_n\geq0$ with $\sum_{i\in[n]}p_i=1$, where for any positive integer $n$, we define $[n]$ to be the set $\{1,\ldots,n\}$. 
Under this setting, \citet{luedtke2010integer,ruspp:02} observed that the joint linear CCP, defined by~\eqref{eq:ccp}, can be reformulated as a mixed-integer linear program  as follows:
\begin{subequations}\label{eq:ccp-re}
\begin{align}
\min\quad& \vh^\top \vx\label{ccp-re-obj}\\
\st\quad& \vx\in\cX\subseteq\R^m,\qquad \vA \vx=\vb+\vy,\\
&y_j\geq w_{i,j}(1-z_i),\quad\forall i\in[n],~\forall j\in[k],\label{ccp-bigM}\\
&\sum_{i\in[n]}p_iz_i\leq \epsilon,\label{ccp-knapsack}\\
&\vy\in\R_+^k,~\vz\in\{0,1\}^n\label{ccp-binary},
\end{align}
\end{subequations}
where $\vb\in\R^k$ is some vector satisfying $\vb(\vxi^i)\geq \vb$ for all $i$ and $\bm{w_i}=(w_{i,1},\ldots,w_{i,k})^\top$ denotes $\vb(\vxi^i)-\vb$. Note that by definition of $\vb$, it follows that the data vector $\bm{w_i}$ is nonnegative for all $i\in[n]$. Observe that $\vA \vx\geq \vb$ are implicit inequalities, due to the chance constraint~\eqref{ccp-cons} with $1-\epsilon>0$. Here, $z_i$ is introduced as an indicator variable to model the event $\vA \vx\geq \vb(\vxi^i)$. More precisely, when $z_i=0$, the constraints~\eqref{ccp-bigM} enforce that $\vy\geq \bm{w_i}$ holds and thus $\vA \vx\geq \vb(\vxi^i)$ is satisfied. On the other hand, when $z_i=1$, it follows that $\vy\geq\bm{0}$ and $\vA \vx\geq \vb$, which is satisfied by default. Therefore,  constraints~\eqref{ccp-bigM}
are referred to as big-$M$ constraints. Finally, \eqref{ccp-knapsack} enforces that the probability of $\vA \vx\geq \vb(\vxi^i)$ being violated is at most~$\epsilon$. 

The size of the  deterministic equivalent formulation of the  joint  CCP given by~\eqref{eq:ccp-re} grows linearly with the number of scenarios. Unfortunately, such a reformulation based on big-M constraints comes with the disadvantage that the corresponding relaxations obtained by relaxing the binary variables into continuous are weak.  Thus, in order to achieve effectiveness in practical implementation, these reformulations must be strengthened with additional valid inequalities. 

A particularly important and widely applicable class of valid inequalities that strengthen the big-M reformulations of CCPs rely on a critical specific substructure in the formulation~\eqref{eq:ccp-re}, called a {\it mixing set with binary variables}; see e.g., \citet{luedtke2010integer} and \citet{kucukyavuz2012mixing}. Formally, given nonnegative coefficients $w_{1,j},\ldots,w_{n,j}$, %
a mixing set with binary variables is defined as follows: 
\[
\text{MIX}_j:=\left\{(y_j,\vz)\in\R_+\times\{0,1\}^n:~ y_j+w_{i,j}z_i \geq w_{i,j},~\forall i\in[n]\right\}; %
\]
hence the set defined by ~\eqref{ccp-bigM} and~\eqref{ccp-binary}, i.e.,
\[
\left\{(\vy,\vz)\in\R_+^k\times\{0,1\}^n:~ y_j+w_{i,j}z_i \geq w_{i,j},~\forall i\in[n],~\forall j\in[k]\right\}, %
\]
is nothing but a joint mixing set that shares common binary variables $\vz$, but independent continuous variables $y_j$, $j\in[k]$. Here, note that the set defined by ~\eqref{ccp-bigM} and~\eqref{ccp-binary} is precisely $\cM(\vW,\bm{\ell},\varepsilon)$ when $\bm{\ell}=\bm{0}$ and $\varepsilon=0$. Also, it is worthwhile to note that the  constraint~\eqref{ccp-knapsack} is a knapsack constraint. Therefore, the formulation~\eqref{eq:ccp-re} can be strengthened by the inclusion of valid inequalities originating from the set defined by \eqref{ccp-bigM}--\eqref{ccp-binary}.

The term {\it mixing set} is originally coined by \citet{gunluk2001mixing} for the sets of the form 
\[
\text{GMIX} := \left\{(y,\vz)\in\R_+\times\Z^n:~ y+u z_i \geq q_i,~\forall i\in[n] \right\},
\]
where the parameters are $u\in\R_+$  %
and $\vq=(q_1,\ldots,q_n)^\top\in\R^n$. 
Such sets $\text{GMIX}$ with general integer variables have applications in lot sizing and capacitated facility location problems; see e.g., ~\cite{conforti2007mix-flow,constantino2010divisible,gunluk2001mixing,miller2003tight,zhao2008div-cap} (see also~\cite{wolsey2010mixing-survey} for a survey of the area). 
For mixing sets with general integer variables such as $\text{GMIX}$ defined above, \citet{gunluk2001mixing} introduced the so-called {\it mixing inequalities}---an exponential family of linear inequalities that admits an efficient separation oracle---and showed that this class of inequalities are sufficient to describe the associated convex hull of the sets $\text{GMIX}$. In fact, prior to \cite{gunluk2001mixing}, in the context of lot-sizing problems, \citet[Theorem 18]{PW94} obtained the same result, albeit without using the naming convention of mixing sets/inequalities. 
Furthermore, the equivalence of $\text{MIX}_j$ and  
$\text{GMIX}$ under the additional domain restrictions $\vz\in\{0,1\}^n$ and the assumption $u\geq \max_i q_i$ is immediate.
The appearance of mixing sets with binary variables dates back to the work of \citet{atamturk2000mixed} on vertex covering. Essentially, it was shown in \cite{atamturk2000mixed} that the intersection of several sets of the form MIX$_j$ with common binary variables $\vz$ but separate continuous variables $y_j$, $j\in [k]$ can be characterized by the intersection of the corresponding {\it star inequalities}; see \cite[Theorem~3]{atamturk2000mixed}. 
Furthermore, it is well-known \cite{luedtke2010integer} that mixing inequalities for $\text{MIX}_j$ are equivalent to the star inequalities introduced in~\cite{atamturk2000mixed}. 
We will give a formal definition of mixing (star) inequalities for mixing sets with binary variables in Section~\ref{sec:mixing}.

Due to the importance of their use in joint CCPs, the mixing (with knapsack) substructure  \eqref{ccp-bigM}--\eqref{ccp-binary} present in the reformulations of joint CCPs has received a lot of attention in the more recent literature.
\begin{itemize}
\item For general $k$, i.e., when the number of linear constraints inside the chance constraint is more than one, \citet{atamturk2000mixed} proved that the convex hull of a joint mixing set of the form~\eqref{ccp-bigM} and~\eqref{ccp-binary}, which is equivalent to $\cM(\vW,\bm{0},0)$, can be described by applying the mixing inequalities. 
\item For $k=1$, \citet{luedtke2010integer}, \citet{kucukyavuz2012mixing}, and \citet{abdi2016mixing-knapsack} suggested valid inequalities for a single mixing set subject to the knapsack constraint~\eqref{ccp-knapsack}.
\item For general $k$, \citet{kucukyavuz2012mixing} and \citet{zhao2017joint-knapsack} proposed valid inequalities for a joint mixing set with a knapsack constraint.
\end{itemize}

\citet{luedtke2010integer} showed that the problem is NP-hard  for $k>1$ even when the  restrictions inside the chance constraints are linear and each scenario has equal probability, in which case the knapsack constraint \eqref{ccp-knapsack} becomes a cardinality constraint.  
However, \citet{kucukyavuz2012mixing} argued that the problem for $k=1$ under equiprobable scenarios is polynomial-time solvable and gave a compact and tight extended formulation based on disjunctive programming. Note that while not explicitly stated in \cite{kucukyavuz2012mixing}, when $k=1$ the polynomial-time solvability argument extends for the unequal probability case.

Many of these prior works aim to convexify a (joint) mixing set with a knapsack constraint directly. In contrast, in our paper we exploit the knapsack structure through an indirect approach based on {\it quantile inequalities}. Given $\vc\in\R_+^k$ and $\delta>0$, the {\it $(1-\delta)$-quantile for $\vc^\top \vy$} is defined as 
\[
q_{\vc,\delta}:=\min \left\{\vc^\top \vy :~\sum_{i\in[n]}p_iz_i\leq\delta,~(\vy,\vz)\text{ satisfies }\eqref{ccp-bigM},\eqref{ccp-binary}\right\}, 
\]
and the inequality $\vc^\top \vy\geq q_{\vc,\delta}$ is called a {\it $(1-\delta)$-quantile cut}. By definition, a $(1-\epsilon)$-quantile cut is valid for the solutions satisfying~\eqref{ccp-bigM}--\eqref{ccp-binary}. The quantile cuts have been studied in~\cite{ahmed2017nonanticipative,luedtke2014branch-and-cut,song2014cc-packing,qiu2014covering-lp,LKL16,xie2018quantile}, and their computational effectiveness %
has been observed in practice. 
As opposed to mixing sets and associated mixing inequalities, the quantile cuts link many continuous variables together; it is plausible to conjecture that this linking of the continuous variables is the one of the main sources of their effectiveness in practice.

The structure of a joint mixing set with lower bounds $\cM(\vW,\vell,\varepsilon)$, defined in~\eqref{eq:general-set}, %
is flexible enough to simultaneously work with quantile cuts. 
For $j\in[k]$, let $\ell_j$ denote the $(1-\epsilon)$-quantile for $\vc^\top \vy= y_j$. Then, for any $j\in[k]$, we have
\[
\ell_j=\min \left\{\max_{i\in[n]}\left\{w_{i,j}(1-z_i)\right\} :~ \vz\text{ satisfies }\eqref{ccp-knapsack},\eqref{ccp-binary}\right\}. 
\]
Note that $\ell_j$ can be computed in $O(n\log n)$ time, because without loss of generality we can assume $w_{1,j}\geq\cdots\geq w_{n,j}$ after possible reordering of $[n]$, and the optimum value of the above optimization problem is precisely $w_{t,j}$ where $t$ is the index such that $\sum_{i\leq t-1}p_i\leq \epsilon$ and $\sum_{i\leq t}p_i>\epsilon$. Although the $(1-\epsilon)$-quantile for $\sum_{j\in[k]}y_j$ seems harder to compute, at least we know that the value is greater than or equal to $\sum_{j\in[k]}\ell_j$. Therefore, we have quantile cuts $y_j\geq \ell_j$ for $j\in[k]$ and $\sum_{j\in[k]}y_j\geq \varepsilon+\sum_{j\in[k]}\ell_j$ for some $\varepsilon\geq0$, and the set defined by these quantile cuts and the constraints \eqref{ccp-bigM},~\eqref{ccp-binary} is precisely a set of the form $\cM(\vW,\vell,\varepsilon)$. Similarly, it is straightforward to capture the quantile cut $\vc^\top \vy\geq \varepsilon+\sum_{j\in[k]}c_j\ell_j$ for general $\vc\in\R_+^k$, because we can rewrite $y_j\geq \ell_j$ for $j\in[k]$, \eqref{ccp-bigM}~and~\eqref{ccp-binary} in terms of $c_1y_1,\ldots,c_jy_j$, and thus the resulting system is equivalent to a joint mixing set with lower bounds. 

 Next, we summarize our contributions and provide an outline of the paper. 

\subsection{Contributions and outline}\label{sec:outline}

In this paper, we study the polyhedral structure of $\cM(\vW,\vell,\varepsilon)$, i.e., joint mixing sets with lower bounds, mainly in the context of joint linear CCPs with random right-hand sides and a discrete probability distribution.
Our approach is based on a connection between mixing sets and submodularity that has been overlooked in the literature. Therefore, in Section~\ref{sec:prelim}, we first discuss basics of submodular functions and polymatroid inequalities as they relate to our work. In addition, we devote Section~\ref{sec:jointSubmodularity} to establish new tools on a particular joint submodular structure; these new tools play a critical role in our analysis of the joint mixing sets. 

Our contributions are as follows: 
\begin{enumerate}[(i)]
\item We first establish a strong and somewhat surprising connection between polymatroids and the basic mixing sets with binary variables   (Section \ref{sec:mixing}). It is well-known that submodularity imposes favorable characteristics in terms of explicit convex hull descriptions via known classes of inequalities and their efficient separation. In particular,  the idea of utilizing polymatroid inequalities from submodular functions has appeared in various papers in other contexts for specific binary integer programs~\cite{atamturk2008polymatroids,atamturk2019conic-quad,baumann2013submodular,xie2018dro,yu2017cardinality,zhang2018drobinary}. Notably, mixing sets have been known to be examples of simple structured sets whose convex hull descriptions possess similar favorable characteristics. However, to the best of our knowledge, the connection between submodularity and mixing sets  has not been recognized before. Establishing this connection enables us to unify and generalize various existing results on mixing sets with binary variables. 

\item In Section \ref{sec:aggregated}, we propose a new class of valid inequalities, referred to as the {\it aggregated mixing inequalities},  for the set $\mathcal M(\vW,\vell,\varepsilon)$. One important feature of the class of aggregated mixing inequalities as opposed to the standard mixing inequalities is that it is specifically designed to  simultaneously exploit the information encoded in multiple mixing sets with common binary variables. 

\item In Section \ref{sec:mixing-lb}, we establish conditions under which  the convex hull of the set $\mathcal M(\vW,\vell,\varepsilon)$ can be characterized through a submodularity lens. We show that the new class of aggregated mixing inequalities, in addition to the classical mixing inequalities, are sufficient under appropriate conditions.

\item In Section \ref{sec:two-sided}, we revisit  the results from a recent paper by \citet{liu2018intersection} on modeling two-sided CCPs. We show that mixing sets of the particular structure considered in~\cite{liu2018intersection} is nothing but a joint mixing set with lower bound structure with $k=2$ and two additional constraints involving only the continuous variables $\vy$. Thus, our results on aggregated mixing inequalities are immediately applicable to two-sided CCPs. In addition, we show that, due to the simplicity of the additional constraints on the variables $\vy$ in two-sided CCPs, our general convex hull results on $\cM(\vW,\vell,\varepsilon)$ can be extended easily to accommodate the additional constraints on $\vy$ and recover the convex hull results from \cite{liu2018intersection}.
\end{enumerate}

Finally, we would like to highlight that although our results are motivated by joint CCPs, they are broadly applicable to other settings where the intersection of mixing sets with common binary variables is present. In addition, applicability of our results from Section~\ref{sec:jointSubmodularity} extend to other cases where epigraphs of general submodular functions appear in a similar structure. %

\subsection{Notation}\label{sec:notation}

Given a positive integer $n$, we let $[n]:=\{1,\ldots,n\}$. 
We let $\bm{0}$ denote the vector of all zeros whose dimension varies depending on the context, and similarly, $\bm{1}$ denotes the vector of all ones. $\bm{e^j}$ denotes the unit vector whose $j\textsuperscript{th}$ coordinate is 1, and its dimension depends on the context.  
For $V\subseteq [n]$, $\bm{1}_V\in\{0,1\}^n$ denotes the characteristic vector, or the incidence vector, of $V$. 
For a set $Q$, we denote its convex hull and the extreme points of its convex hull by $\conv(Q)$ and $\ext(Q)$ respectively. 
For $t\in\R$, $\left(t\right)_+$ denotes $\max\{0,t\}$. 
Given a vector $\vpi\in\R^n$, and a set $V\subseteq [n]$, we define $\vpi(V)=\sum_{i\in V}\pi_i$.  
 For notational purposes, when $S=\emptyset$, we define $\max_{i\in S} s_i=0$ and $\sum_{i\in S} s_i=0$.

\section{Submodular functions and polymatroid inequalities}\label{sec:submodular}
In this section, we start with a brief review of submodular functions and polymatroid inequalities, and then in Section~\ref{sec:jointSubmodularity} we establish tools on joint submodular constraints that are useful for our analysis of $\cM(\vW,\vell,\varepsilon)$.

\subsection{Preliminaries}\label{sec:prelim}
Consider an integer $n\geq 1$ and a set function $f:2^{[n]}\to\R$. Recall that $f$ is {\it submodular} if
\[
f(A)+f(B)\geq f(A\cup B)+f(A\cap B),\quad \forall A,B\subseteq [n].
\]
Given a submodular set function $f$, \citet{edmonds1970polymatroid} introduced the notion of {\it extended polymatroid of $f$}, which is a polyhedron associated with $f$ defined as follows: 
\begin{equation}\label{ext-polymat}
EP_f:=\left\{\vpi\in\R^n:~\vpi(V)\leq f(V),~\forall V\subseteq[n]\right\}.
\end{equation}
Observe that $EP_f$ is nonempty if and only if $f(\emptyset)\geq0$. In general, a submodular function $f$ need not satisfy $f(\emptyset)\geq 0$. Nevertheless, it is straightforward to see that the function $f-f(\emptyset)$ is submodular whenever $f$ is submodular, and that $(f-f(\emptyset))(\emptyset)=0$. Hence, $EP_{f-f(\emptyset)}$ is always nonempty. Hereinafter, we use notation $\tilde f$ to denote $f-f(\emptyset)$ for any set function $f$.

A function on $\{0,1\}^n$ can be interpreted as a set function over the subsets of $[n]$, and thus, the definitions of submodular functions and extended polymatroids extend to functions over $\{0,1\}^n$. To see this, consider any integer $n\geq 1$ and any function $f:\{0,1\}^n\to\R$. With a slight abuse of notation, define $f(V):=f(\bm{1}_V)$ for $V\subseteq[n]$ where $\bm{1}_V$ denotes the characteristic vector of $V$. We say that $f:\{0,1\}^n\to\R$ is a submodular function if the corresponding set function over $[n]$ is submodular. We can also define the extended polymatroid of $f:\{0,1\}^n\to\R$ as in~\eqref{ext-polymat}. Throughout this paper, given a function $f:\{0,1\}^n\to\R$, we will switch between its set function interpretation and its original form, depending on the context.

Given a submodular function $f:\{0,1\}^n\to\R$, its {\it epigraph} is the mixed-integer set given by
\[
Q_f=\left\{(y,\vz)\in\R\times\{0,1\}^n:~y\geq f(\vz)\right\}.
\]
It is well-known that when $f$ is submodular, one can characterize the convex hull of $Q_f$ through the extended polymatroid of $\tilde f$.

\begin{theorem}[{\citet{lovasz1983submodular}, \citet[Proposition 1]{atamturk2008polymatroids}}]\label{thm:lovasz}
	Let $f:\{0,1\}^n\to\R$ be a submodular function. Then 
	\[
	\conv(Q_f)=\left\{(y,\vz)\in\R\times[0,1]^n:~y\geq \vpi^\top \vz+f(\emptyset),~\forall \vpi\in EP_{\tilde f}\right\}.
	\]
\end{theorem}

The inequalities $y\geq \vpi^\top \vz+f(\emptyset)$ for $\vpi\in EP_{\tilde f}$ are called the {\it polymatroid inequalities of~$f$}. Although there are infinitely many polymatroid inequalities of $f$, for the description of $\conv(Q_f)$, it is sufficient to consider only the ones corresponding to the extreme points of $EP_{\tilde f}$. We refer to the polymatroid inequalities defined by the extreme points of $EP_{\tilde f}$ as the {\it extremal polymatroid inequalities of $f$}. Moreover, \citet{edmonds1970polymatroid} provided the following explicit characterization of the extreme points of $EP_{\tilde f}$. 
\begin{theorem}[\citet{edmonds1970polymatroid}]\label{thm:edmonds}
Let $f:\{0,1\}^n\to\R$ be a submodular function. Then $\vpi\in\R^n$ is an extreme point of $EP_{\tilde f}$ if and only if there exists a permutation $\sigma$ of~$[n]$ such that $\pi_{\sigma(t)}=f(V_{t})-f(V_{t-1})$, where $V_{t}=\{\sigma(1),\ldots,\sigma(t)\}$ for $t\in[n]$ and $V_{0}=\emptyset$.
\end{theorem}

The algorithmic proof of Theorem~\ref{thm:edmonds} from \citet{edmonds1970polymatroid} is of interest. Suppose that we are given a linear objective $\bm{\bar z}\in\R^n$; then $\max_{\vpi}\left\{\bm{\bar z}^\top \vpi:~\vpi\in EP_{\tilde f}\right\}$ can be solved by the following ``greedy" algorithm: given $\bm{\bar z}\in\R^n$, first find an ordering $\sigma$ such that $\bar z_{\sigma(1)}\geq\cdots\geq\bar z_{\sigma(n)}$, and let $V_{t}:=\{\sigma(1),\ldots,\sigma(t)\}$ for $t\in[n]$ and $V_{0}=\emptyset$. Then, $\vpi\in\R^n$ where $\pi_{\sigma(t)}=f(V_{t})-f(V_{t-1})$ for $t\in[n]$ is an optimal solution to $\max_{\vpi}\left\{\bm{\bar z}^\top \vpi:~\vpi\in EP_{\tilde f}\right\}$. Note that the implementation of this algorithm basically requires a sorting algorithm to compute the desired ordering $\sigma$, and this can be done in $O(n\log n)$ time. Thus,  the overall complexity of this algorithm is $O(n\log n)$. 

Consequently, given a point $({\bar y},\bm{\bar z})\in\R\times \R^n$, separating a violated polymatroid inequality amounts to solving the optimization problem $\max_{\vpi}\left\{\bm{\bar z}^\top \vpi :~\vpi\in EP_{\tilde f}\right\}$, and thus we arrive at the following result.
\begin{corollary}[{\citet[Section 2]{atamturk2008polymatroids}}]\label{cor:separation}
	Let $f:\{0,1\}^n\to\R$ be a submodular function. Then the separation problem for polymatroid inequalities can be solved in $O(n\log n)$ time.
\end{corollary}

\subsection{Joint submodular constraints}\label{sec:jointSubmodularity}

In this section, we establish tools that will be useful throughout this paper. Recall that when $f$ is submodular, the convex hull of its epigraph $Q_f$ is described by the extremal polymatroid inequalities of $f$. Henceforth, we use the restriction $(y,\vz)\in \conv(Q_f)$ as a constraint to indicate the inclusion of the corresponding extremal polymatroid inequalities of $f$ in the constraint set.

Let $f_1,\ldots,f_k:\{0,1\}^n\to\R$ be $k$ submodular functions. Let us examine the convex hull of the following mixed-integer set:
\[
Q_{f_1,\ldots,f_k}:=\left\{(\vy,\vz)\in\R^k\times\{0,1\}^n:~y_1\geq f_1(\vz),\ldots,y_k\geq f_k(\vz)\right\}.
\]
When $k=1$, the set $Q_{f_1}$ is just the epigraph of the submodular function $f_1$ on $\{0,1\}^n$. For general $k$, $Q_{f_1,\ldots,f_k}$ is described by $k$ submodular functions that share the same set of binary variables. For $(\vy,\vz)\in Q_{f_1,\ldots,f_k}$, constraint $y_j\geq f_j(\vz)$ can be replaced with $(y_j,\vz)\in Q_{f_j}$ for $j\in[k]$. Therefore, the polymatroid inequalities of $f_j$ with left-hand side $y_j$, of the form $y_j\geq \vpi^\top  \vz+f_j(\emptyset)$ with $\vpi\in EP_{\tilde{f_j}}$, are valid for $Q_{f_1,\ldots,f_k}$. In fact, these inequalities are sufficient to describe $\conv(Q_{f_1,\ldots,f_k})$ as well. 

\begin{proposition}[{\citet[Theorem 2]{baumann2013submodular}}]\label{prop:intersection1}
Let the functions $f_1,\ldots,f_k:\{0,1\}^n\rightarrow\R$ be submodular. Then, 
\[
\conv\left(Q_{f_1,\ldots,f_k}\right)=\left\{(\vy,\vz)\in\R^k\times[0,1]^n:~(y_j,\vz)\in \conv(Q_{f_j}),~\forall j\in[k]\right\}.
\]
\end{proposition}

By Proposition~\ref{prop:intersection1}, when $f_1,\ldots,f_k$ are submodular, $\conv\left(Q_{f_1,\ldots,f_k}\right)$ can be described by the polymatroid inequalities of $f_j$ with left-hand side $y_j$ for $j\in[k]$. The submodularity requirement on all of the functions $f_j$ in Proposition~\ref{prop:intersection1} is indeed critical. We demonstrate in the next example that even when $k=2$, and only one of the functions $f_i$ is not submodular, we can no longer describe the corresponding convex hull using the polymatroid inequalities for $f_j$. 

\begin{example}\label{ex:submodularityNecessity}
Let $f_1,f_2:\{0,1\}^2\rightarrow\R$ be defined by 
\[
f_1(0,0)=f_1(1,1)=0,~f_1(0,1)=f_1(1,0)=1\quad\text{and}\quad f_2(0,0)=f_2(1,1)=1,~f_2(0,1)=f_2(1,0)=0.
\]
While $f_1$ is submodular, $f_2$ is not. Since $f_1(0,0)=f_1(1,1)=0$, we deduce that $(0,1/2,1/2)\in \conv(Q_{f_1})$. Similarly, as $f_2(0,1)=f_2(1,0)=0$, it follows that $(0,1/2,1/2)\in \conv(Q_{f_2})$. This implies that 
\[
(0,0,1/2,1/2)\in\left\{(\vy,\vz)\in\R^2\times[0,1]^2:~(y_1,\vz)\in \conv(Q_{f_1}),~(y_2,\vz)\in \conv(Q_{f_2})\right\}.
\]
Notice that, by definition of $f_1,f_2$, we have $f_1(\vz)+f_2(\vz)=1$ for each $\vz\in\{0,1\}^2$, implying in turn that $y_1+y_2\geq 1$ is valid for $\conv\left(Q_{f_1,f_2}\right)$. Therefore, the point $(0,0,1/2,1/2)$ cannot be in $\conv\left(Q_{f_1,f_2}\right)$. So, it follows that $\conv\left(Q_{f_1,f_2}\right)\neq \left\{(\vy,\vz)\in\R^2\times[0,1]^2:~(y_j,\vz)\in \conv(Q_{f_j}),~\forall j\in[2]\right\}$.
\end{example}

In Section~\ref{sec:mixing}, we will discuss how Proposition~\ref{prop:intersection1} can be used to provide the convex hull description of a joint mixing set $\cM(\vW,\bm{0},0)$. 

We next highlight a slight generalization of Proposition~\ref{prop:intersection1} that is of interest for studying $\cM(\vW,\vell,\varepsilon)$. Observe that $Q_{f_1,\ldots,f_k}$ is defined by multiple submodular constraints with independent continuous variables $y_j$. We can replace this independence condition by a certain type of dependence. Consider the following mixed-integer set:
\begin{equation}\label{eq:mult-dep-sub}
\cP=\left\{(\vy,\vz)\in \R^k\times \{0,1\}^n:~\bm{a_1}^\top \vy \geq f_1(\vz),\ldots,\bm{a_m}^\top \vy\geq f_m(\vz)\right\}
\end{equation}
where $\bm{a_1},\ldots,\bm{a_m}\in\R_+^k\setminus\{\bm{0}\}$ and $f_1,\ldots, f_m:\{0,1\}^n\rightarrow\R$ are submodular functions. Here, $m$ can be larger than $k$, so $\bm{a_1},\ldots,\bm{a_m}$ need not be linearly independent. Now consider $\valpha=\sum_{j\in[m]}c_j \bm{a_j}$ for some $\vc\in\R_+^m$. Notice that $f_{\valpha}\geq \sum_{j\in[m]}c_jf_j$ where $f_{\valpha}:\{0,1\}^n\to\R$ is defined as
\begin{equation}\label{lb-alphay}
f_{\valpha}(\vz):=\min\left\{\valpha^\top \vy:~(\vy,\vz)\in \mathcal P\right\},\quad\forall \vz\in\{0,1\}^n.
\end{equation}
\begin{definition}\label{def:weakindep}
We say that $\bm{a_1}^\top\vy,\ldots,\bm{a_m}^\top\vy$ are {\it weakly independent with respect to $f_1,\ldots,f_m$} if for any $\valpha=\sum_{j\in[m]}c_j\bm{a_j}$ with $\vc\in\R_+^m$, we have $f_{\valpha}=\sum_{j\in[m]}c_jf_j$. 
\end{definition}
It is straightforward to see that if $\bm{a_1},\ldots,\bm{a_m}$ are distinct unit vectors, i.e., $m=k$ and $\bm{a_j}^\top \vy=y_j$ for $j\in[k]$, then $\bm{a_1}^\top \vy,\ldots, \bm{a_m}^\top \vy$ are weakly independent. It is also easy to see that if $\bm{a_1},\ldots,\bm{a_m}$ are linearly independent, then $\bm{a_1}^\top  \vy,\ldots,\bm{a_m}^\top \vy$ are weakly independent. Based on this definition, we have the following slight extension of Proposition~\ref{prop:intersection1}.

\begin{proposition}\label{prop:intersection2}
Let $\cP$ be defined as in~\eqref{eq:mult-dep-sub}. If $\bm{a_1}^\top \vy,\ldots, \bm{a_m}^\top \vy$ are weakly independent with respect to $f_1,\ldots,f_m$, then
\[
\conv\left(\cP\right)=\left\{(\vy,\vz)\in\R^k\times[0,1]^n:~(\bm{a_j}^\top \vy,\vz)\in \conv(Q_{f_j}),~\forall j\in[m]\right\}.
\]
\end{proposition}
\begin{proof}
Define $\cR:=\left\{(\vy,\vz)\in\R^k\times[0,1]^n:~(\bm{a_j}^\top \vy,\vz)\in \conv(Q_{f_j}),~\forall j\in[m]\right\}$. 
It is clear that $\conv\left(\cP\right)\subseteq \cR$. For the direction $\conv\left(\cP\right)\supseteq \cR$, we need to show that any inequality $\valpha^\top \vy+{\vbeta}^\top \vz\geq\gamma$ valid for $\conv\left(\cP\right)$ is also valid for $\cR$. To that end, take an inequality $\valpha^\top \vy+{\vbeta}^\top \vz\geq\gamma$ valid for $\conv\left(\cP\right)$. Note that every recessive direction of $\conv(\cP)$ is of the form $(\vr,\bm{0})$ for some $\vr\in\R^k$. Moreover, $(\vr,\bm{0})$ is a recessive direction of $\conv(\cP)$ if and only if $\vr$ satisfies $\bm{a_j}^\top \vr \geq 0$ for all $j\in[m]$. Since $\valpha^\top \vy+{\vbeta}^\top \vz\geq\gamma$ is valid for $\conv\left(\cP\right)$, $\valpha^\top \vr\geq 0$ for every recessive direction $(\vr,\bm{0})$ of $\conv(\cP)$, and therefore, $\valpha^\top \vr\geq 0$ holds for all $\vr\in\{\vr\in\R^k:\bm{a_j}^\top \vr \geq 0,~\forall j\in[m]\}$. Then, by Farkas' lemma, there exists some $\vc\in\R_+^m$ such that $\valpha=\sum_{j\in[m]}c_j\bm{a_j}$.
Moreover, $\valpha^\top \vy+{\vbeta}^\top \vz\geq \gamma$ is valid for
\[
\cQ:=\left\{(\vy,\vz)\in\R^k\times\{0,1\}^n:~\valpha^\top \vy\geq f_{\valpha}(\vz)\right\},
\]
where $f_{\valpha}$ is defined as in~\eqref{lb-alphay}. Since $\bm{a_1}^\top \vy,\ldots, \bm{a_m}^\top \vy$ are weakly independent with respect to $f_1,\ldots,f_m$, it follows that $f_{\valpha}=\sum_{j\in[m]}c_jf_j$, and therefore, $f_{\valpha}$ is submodular. Then it is not difficult to see that 
\[
\conv(\cQ)=\left\{(\vy,\vz)\in\R^k\times[0,1]^n:~(\valpha^\top \vy, \vz)\in\conv(Q_{f_{\valpha}})\right\}.
\]
Therefore, to show that $\valpha^\top \vy+{\vbeta}^\top \vz\geq \gamma$ is valid for $\cR$, it suffices to argue that $\cR\subseteq \conv(\cQ)$. Let $(\bm{\bar y},\bm{\bar z})\in \cR$. Then, by Theorem~\ref{thm:lovasz}, it suffices to show that $\valpha^\top \bm{\bar y}\geq \vpi^\top \bm{\bar z} +f_{\valpha}(\emptyset)$ holds for every extreme point $\vpi$ of $EP_{\tilde{f_{\valpha}}}$. To this end, take an extreme point $\vpi$ of $EP_{\tilde{f_{\valpha}}}$. By Theorem~\ref{thm:edmonds}, there exists a permutation $\sigma$ of $[n]$ such that $\pi_{\sigma(t)}=f_{\valpha}(V_{t})-f_{\valpha}(V_{t-1})$ where $V_{t}=\{\sigma(1),\ldots,\sigma(t)\}$ for $t\in[n]$ and $V_{0}=\emptyset$. Now, for $j\in[m]$, let $\bm{\vpi^j}\in\R^n$ be the vector such that $\bm{\vpi^j}_{\sigma(t)}=f_j(V_t)-f_j(V_{t-1})$ for $t\in[n]$. Then, we have $\vpi=\sum_{j\in[m]}c_j\bm{\vpi^j}$ because $f_{\valpha}=\sum_{j\in[m]}c_jf_j$. Moreover, by Theorem~\ref{thm:edmonds}, $\bm{\vpi^j}$ is an extreme point of $EP_{\tilde{f_j}}$. Hence, due to our assumption that $(\bm{a_j}^\top \bm{\bar y},\bm{\bar z})\in \conv(Q_{f_j})$, Theorem~\ref{thm:lovasz} implies $\bm{a_j}^\top \bm{\bar y}\geq(\bm{\vpi^j})^\top\bm{\bar z}+\pi_j(\emptyset)$ is valid for all $j\in[m]$. Since $\valpha^\top \bm{\bar y}\geq \vpi^\top \bm{\bar z} +f_{\valpha}(\emptyset)$ is obtained by adding up $\bm{a_j}^\top \bm{\bar y}\geq(\bm{\vpi^j})^\top \bm{\bar z}+\pi_j(\emptyset)$ for $j\in[m]$, it follows that $\valpha^\top \bm{\bar y}\geq \vpi^\top \bm{\bar z} +f_{\valpha}(\emptyset)$ is valid, as required. We just have shown that $\cR\subseteq \conv(\cQ)$, thereby completing the proof.
\end{proof}

In Section~\ref{sec:mixing-lb}, we will use Proposition~\ref{prop:intersection2} to study  the convex hull of $\cM(\vW,\mathbf 0,\varepsilon)$, i.e., a joint mixing set with a linking constraint. Again, the submodularity assumption on $f_1,\ldots,f_m$ is important in Proposition~\ref{prop:intersection2}. 
Recall that Example~\ref{ex:submodularityNecessity} demonstrates that in Proposition~\ref{prop:intersection2} even when $m$ is taken to be equal to $k$ and the vectors $\bm{a_j}\in \R_+^k\setminus\{\bf{0}\}$, $j\in[m]=[k]$, are taken to be the unit vectors in $\R^k$, the statement does not hold if one of the functions $f_j$ is not submodular.

\section{Mixing inequalities and joint mixing sets}\label{sec:mixing}
In this section, we establish that mixing sets with binary variables are indeed nothing but the epigraphs of certain submodular functions. In addition, through this submodularity lens, we prove that the well-known mixing (or star) inequalities for mixing sets are nothing but the extremal polymatroid inequalities.

Recall that a joint mixing set with lower bounds $\cM(\vW,\vell,\varepsilon)$, where $\vW\in\R^{n\times k}_+$, $\vell\in\R_+^k$ 
and $\varepsilon\geq 0$, is defined by~\eqref{eq:general-set}. In this section, we study the case when $\varepsilon=0$, and characterize the convex hull of $\cM(\vW,\vell,0)$ for any $\vW\in\R^{n\times k}_+$ and $\vell\in\R^k_+$. As corollaries, we prove that the famous star/mixing inequalities are in fact polymatroid inequalities, and we recover the result of  \citet[Theorem 3]{atamturk2000mixed} on joint mixing sets $\cM(\vW,\bm{0},0)$. 

Given a matrix $\vW=\{w_{i,j}\}\in\R^{n\times k}_+$ and a vector $\vell\in\R_+^k$, we define the following mixed-integer set:
\begin{equation}\label{twisted-mixingset}
\cP(\vW,\vell,\varepsilon)=\left\{(\vy,\vz)\in\R^k\times\{0,1\}^n:~
\eqref{bigM''}\--\eqref{linking''}
\right\}
\end{equation}
where
\begin{align}
&y_j\geq w_{i,j}z_i,&\forall i\in[n],~j\in[k], \label{bigM''}\\
&y_j\geq \ell_j,&\forall j\in[k], \label{lb''} \\
&\sum_{j\in[k]} y_j\geq \varepsilon + \sum_{j\in[k]}\ell_j. \label{linking''} 
\end{align}
\begin{remark}\label{rem:M-P-relation}
By definition, $(\vy,\vz)\in \cM(\vW,\vell,\varepsilon)$ if and only if $(\vy,\bm{1}-\vz)\in \cP(\vW,\vell,\varepsilon)$. 
Thus, the convex hull of $\cM(\vW,\vell,\varepsilon)$ can be obtained after taking the convex hull of $\cP(\vW,\vell,\varepsilon)$ and complementing the $z$ variables. 
\end{remark}

For $j\in[k]$, we define
\begin{equation}\label{y_j-submod}
f_j(\vz):=\max\left\{\ell_j,~\max\limits_{i\in[n]}\left\{w_{i,j}z_i\right\}\right\},\quad\forall \vz\in\{0,1\}^n.
\end{equation}
Then, the set $\cP(\vW,\vell,0)$ admits a representation as the intersection of epigraphs of the functions $f_j(\vz)$:
\[
\cP(\vW,\vell,0)=\left\{(\vy,\vz)\in\R^k\times\{0,1\}^n:~y_j\geq f_j(\vz),~\forall j\in[k]\right\}.
\]

We next establish that the functions $f_j(\vz), j\in[k]$ are indeed submodular. 
\begin{lemma}\label{lem:submodular1}
Let $\vell\in\R_+^k$. For each $j\in[k]$, the function $f_j$ defined as in~\eqref{y_j-submod} satisfies $f_j(\emptyset)=\ell_j$  and it is submodular.
\end{lemma}
\begin{proof}
Let $j\in[k]$. %
Notice that $f_j(\emptyset)=f_j(\bm{0})=\max\left\{\ell_j,0\right\}=\ell_j$. In order to establish the submodularity of $f_j$, for ease of notation, we drop the index $j$ and use $f$ to denote $f_j$. As before, for each $V\subseteq[n]$, let $f(V)$ be defined as $f(\bm{1}_V)$ where $\bm{1}_V\in\{0,1\}^n$ denotes the characteristic vector of $V$. Consider two sets $U,V\subseteq[n]$. By definition of $f$, we have $\max\{f(U),f(V)\}=f(U\cup V)$, and $\min\{f(U),f(V)\}\geq f(U\cap V)$. Then we immediately get
\[
f(U) + f(V) = \max\{f(U),f(V)\} + \min\{f(U),f(V)\} \geq f(U\cup V) + f(U\cap V),
\] 
thereby proving that $f_j$ is submodular, as required.
\end{proof}

\begin{corollary}\label{cor:intersection1}
Let $\vell\in\R^k_+$ and $f_j$ be as defined in \eqref{y_j-submod}. Then,  
\begin{equation*}\label{joint-mixing}
\conv(\cM(\vW,\vell,0))=\left\{(\vy,\vz)\in\R^k\times[0,1]^n:~(y_j,\bm{1}-\vz)\in\conv(Q_{f_j}),~\forall j\in[k]\right\},
\end{equation*}
i.e., the convex hull of $\cM(\vW,\vell,0)$ is given by the extremal polymatroid inequalities of particular submodular functions.
\end{corollary}
\begin{proof}
We deduce from Proposition~\ref{prop:intersection1} that
\[
\conv(\cP(\vW,\vell,0))=\left\{(\vy,\vz)\in\R^k\times[0,1]^n:~(y_j,\vz)\in\conv(Q_{f_j}),~\forall j\in[k]\right\},
\]
which immediately implies the desired relation via Remark~\ref{rem:M-P-relation} and Theorem~\ref{thm:lovasz} since the constraint $(y_j,\bm{1}-\vz)\in\conv(Q_{f_j})$ is equivalent to the set of the corresponding extremal polymatroid inequalities.
\end{proof}
Corollary~\ref{cor:intersection1} establishes a strong connection between the mixing sets with binary variables and the epigraphs of submodular functions, and implies that the convex hull of joint mixing sets are given by the extremal polymatroid inequalities. To the best of our knowledge this connection between mixing sets with binary variables and submodularity has not been identified in the literature before.

An explicit characterization of the convex hull of a mixing set with binary variables in the original space has been studied extensively in the literature. Specifically, \citet{atamturk2000mixed} gave the explicit characterization of $\conv(\cM(\vW,\bm{0},0))$ in terms of the so called {\it mixing (star) inequalities}. Let us state the definition of these inequalities here. 
\begin{definition}\label{def:j-mixing-seq}
We call a sequence $\{j_1\to\cdots \to j_\tau\}$ of indices in $[n]$ a {\it $j$-mixing-sequence} if $w_{j_1,j}\geq w_{j_2,j}\geq\cdots\geq w_{j_\tau, j}\geq \ell_j$. 
\end{definition}
For $\vW=\{w_{i,j}\}\in\R^{n\times k}_+$ and $\vell\in\R_+^k$, the {\it mixing inequality derived from a $j$-mixing-sequence $\{j_1\to\cdots \to j_\tau\}$} is defined as the following (see \cite[Section~2]{gunluk2001mixing}):
\begin{equation}\label{mixing-ineqs}
y_j+\sum_{s\in[\tau]}(w_{j_s,j}-w_{j_{s+1},j})z_{j_s}\geq w_{j_1,j}\tag{$\text{Mix}_{\vW,\vell}$},
\end{equation}
where $w_{j_{\tau+1},j}:=\ell_j$ for convention.  
\citet[Proposition~3]{atamturk2000mixed} showed that the inequality~\eqref{mixing-ineqs} for any $j$-mixing-sequence $\{j_1\to\cdots \to j_\tau\}$ is valid for $\cM(\vW,\vell,0)$ when $\vell=\bm{0}$. 
\citet[Theorem 2]{luedtke2014branch-and-cut} later observed that the inequality~\eqref{mixing-ineqs} for any $j$-mixing-sequence $\{j_1\to\cdots \to j_\tau\}$ is valid for $\cM(\vW,\vell,0)$ for any $\vell\in\R^k_+$. 

Given these results from the literature on the convex hull characterizations of mixing sets and Corollary~\ref{cor:intersection1}, it is plausible to think that there must be a strong connection between the extremal polymatroid inequalities and the mixing (star) inequalities~\eqref{mixing-ineqs}. 
We next argue that the extremal polymatroid inequalities given by the constraint $(y_j,\bm{1}-\vz)\in\conv(Q_{f_j})$ are precisely the mixing (star) inequalities.

\begin{proposition}\label{prop:coeff1}
	Let $\vW=\{w_{i,j}\}\in\R^{n\times k}_+$ and $\vell\in\R^k_+$. Consider any $j\in[k]$. Then, for every extreme point $\vpi$ of $EP_{\tilde{f}_j}$, there exists a $j$-mixing-sequence $\{j_1\to\cdots \to j_\tau\}$ in $[n]$ that satisfies the following:
	\begin{enumerate}[(1)]
		\item $w_{j_1,j}=\max\left\{w_{i,j}:i\in[n]\right\}$,
		\item the corresponding polymatroid inequality $y_j +\sum_{i\in[n]}\pi_i z_i \geq \ell_j +\sum_{i\in[n]}\pi_i$ is equivalent to
		the mixing inequality~\eqref{mixing-ineqs} derived from the sequence $\{j_1\to\cdots \to j_\tau\}$.
	\end{enumerate}
	In particular, for any $j\in[k]$, the extremal polymatroid inequality is of the form
	\begin{equation}\label{extremal-mixing}
	y_j+\sum_{s\in[\tau]}(w_{j_s,j}-w_{j_{s+1},j})z_{j_s}\geq \max\left\{w_{i,j}:i\in[n]\right\}, \tag{$\text{Mix}_{\vW,\vell}^*$}
	\end{equation}
	where $w_{j_1,j}=\max\left\{w_{i,j}:i\in[n]\right\}$ and $w_{j_{\tau+1},j}:=\ell_j$. 
\end{proposition}
\begin{proof}
By Theorem~\ref{thm:edmonds}, there exists a permutation $\sigma$ of~$[n]$ such that $\pi_{\sigma(t)}=f_j(V_{t})-f_j(V_{t-1})$ where $V_{t}=\{\sigma(1),\ldots,\sigma(t)\}$ for $t\in[n]$ and $V_{0}=\emptyset$. By definition of $f_j$ in~\eqref{y_j-submod}, we have $\ell_j=f_j(V_0)\leq f_j(V_1)\leq \cdots\leq f_j(V_n)$, because $\emptyset=V_0\subset V_1\subset\cdots\subset V_n$. Let $\{t_1,\ldots,t_\tau\}$ be the collection of all indices $t$ satisfying $f_j(V_{t-1})<f_j(V_t)$. Without loss of generality, we may assume that $w_{\sigma(t_1),j}\geq\cdots\geq w_{\sigma(t_\tau), j}$. Notice that $w_{\sigma(t_\tau), j}>\ell_j$, because $f_j(V_{t_\tau})>f_j(V_{t_{\tau}-1})\geq\ell_j$. Then, after setting $j_s=\sigma(t_s)$ for $s\in[\tau]$, it follows that $\{j_1\to\cdots \to j_\tau\}$ is a $j$-mixing-sequence. Moreover, we have $w_{j_1,j}=f_j(V_{t_1})=f_j([n])=\max\left\{w_{i,j}:i\in[n]\right\}$. Therefore, we deduce that $\pi_i=w_{j_s,j}-w_{j_{s+1},j}$ if $i=\sigma(t_s)=j_s$ for some $s\in[\tau]$ and $\pi_i=0$ otherwise. 
\end{proof}
As the name "mixing" inequalities is more commonly used in the literature than "star" inequalities, we will stick to the term "mixing" hereinafter to denote the inequalities of the form~\eqref{mixing-ineqs} or~\eqref{extremal-mixing}.

Proposition~\ref{prop:intersection1} and consequently Corollary~\ref{cor:intersection1} imply that, for any facet defining inequality of the set $\conv(\cM(\vW,\vell,0))$, there is a corresponding extremal polymatroid inequality. Proposition~\ref{prop:coeff1} implies that mixing inequalities are nothing but the extremal polymatroid inequalities. 
Therefore, an immediate consequence of Corollary~\ref{cor:intersection1} and Proposition~\ref{prop:coeff1} is the following result.
\begin{theorem}\label{thm:mixing-lb}
Given $\vW=\{w_{i,j}\}\in\R^{n\times k}_+$ and any $\vell\in\R^k_+$, the convex hull of $\cM(\vW,\vell,0)$ is described by the mixing inequalities of the form~\eqref{extremal-mixing} for $j\in[k]$ and the bounds $\bm{0}\leq \vz\leq\bm{1}$.
\end{theorem}

A few remarks are in order. 
\begin{remark}\label{rem:mixingLiterature}
First, note that \citet[Theorem 2]{luedtke2010integer} showed the validity of inequality \eqref{extremal-mixing} and its facet condition for a particular choice of $\vell\in\R^k_+$ in the case of $k=1$.   Also, recall that $\cM(\vW,\bm{0},0)$ is called a joint mixing set, and \citet[Theorem 3]{atamturk2000mixed} proved that $\conv(\cM(\vW,\bm{0},0))$ is described by the mixing inequalities and the bound constraints $\vy\geq \bm{0}$ and $\vz\in[0,1]^n$. Since  Theorem~\ref{thm:mixing-lb} applies to $\cM(\vW,\vell,0)$ for arbitrary $\vell$, it immediately extends \cite[Theorem 3]{atamturk2000mixed} and further extends the validity inequality component of \cite[Theorem 2]{luedtke2010integer}. 
\end{remark}

\begin{remark}\label{rem:mixingSep}
Our final remark is that, since the mixing inequalities~\eqref{extremal-mixing} for $j\in[k]$ are polymatroid inequalities, they can be separated in $O(k\,n\log n)$ time by a simple greedy algorithm, thanks to Corollary~\ref{cor:separation}. This also matches the best known separation complexity of mixing inequalities \cite{gunluk2001mixing}.
\end{remark}

\section{Aggregated mixing inequalities}\label{sec:aggregated}

As discussed in Section~\ref{sec:intro}, in order to make use of the knapsack constraint in the MIP formulation of joint CCPs via quantile cuts, we need to study the set $\cM(\vW,\vell,\varepsilon)$ for general $\varepsilon\geq 0$. Unfortunately, in contrast to our results in Section~\ref{sec:mixing} for the convex hull of $\cM(\vW,\vell,0)$, the convex hull of $\cM(\vW,\vell,\varepsilon)$ for general $\varepsilon\geq 0$ may be complicated; we will soon see this in Example~\ref{example1}. In this section, we introduce a new class of valid inequalities for $\cM(\vW,\vell,\varepsilon)$ for arbitrary $\varepsilon\geq 0$. In Sections~\ref{sec:poly-agg} and~\ref{sec:char}, we identify conditions under which these new inequalities along with the original mixing inequalities are sufficient to give the complete convex hull characterization.

For general $\varepsilon\geq 0$, $\cM(\vW,\vell,\varepsilon)$, given by~\eqref{eq:general-set}, is a subset of $\cM(\vW,\vell,0)$, which means that any inequality valid for $\cM(\vW,\vell,0)$ is also valid for $\cM(\vW,\vell,\varepsilon)$. In particular, Theorem~\ref{thm:mixing-lb} implies that the mixing inequalities of the form~\eqref{mixing-ineqs} are valid for $\cM(\vW,\vell,\varepsilon)$.  However, unlike the $\varepsilon=0$ case, we will see that the mixing inequalities are not sufficient to describe the convex hull of $\cM(\vW,\vell,\varepsilon)$ if $\varepsilon>0$. 

We first present a simplification of $\cM(\vW,\vell,\varepsilon)$. 
Although it is possible that $w_{i,j}<\ell_j$ for some $i,j$ when $\vW,\vell$ are arbitrary, we can reduce $\cM(\vW,\vell,\varepsilon)$ to a set of the form $\cM({\vW}^{\vell},\bm{0},\varepsilon)$ for some ${\vW}^{\vell}=\left\{w^{\vell}_{i,j}\right\}\in\R_+^{n\times k}$.

\begin{lemma}\label{lem:reduction}
Let $\vell\in\R^k_+$. Then 
$\cM(\vW,\vell,\varepsilon)=\left\{(\vy,\vz)\in\R^k\times\R^n:~(\vy-\vell,\vz)\in\cM({\vW}^{\vell},\bm{0},\varepsilon) \right\}$, where ${\vW}^{\vell}=\left\{w_{i,j}^{\vell}\right\}\in\R_+^{n\times k}$ is the matrix whose entries are given by
\begin{equation*}\label{eq:W-nonnegative}
w_{i,j}^{\vell}=(w_{i,j}-\ell_j)_+\quad%
\forall i\in[n],~j\in[k].
\end{equation*}
\end{lemma}
\begin{proof}
By definition, $(\vy-\vell,\vz)\in\cM({\vW}^{\vell},\bm{0},\varepsilon)$ if and only if
	\begin{equation}\label{reduced-con}
	y_j+(w_{i,j}-\ell_j)_+z_i\geq \ell_j+(w_{i,j}-\ell_j)_+,\quad\forall i\in[n],~j\in[k],
	\end{equation}
	and $(\vy,\vz)$ satisfies~\eqref{lb}--\eqref{binary}. Consider any $j\in[k]$. If $\ell_j> w_{i,j}$, then the constraint~\eqref{reduced-con} becomes $y_j\geq \ell_j$ and the inequality $y_j+w_{i,j}z_i\geq w_{i,j}$ is a consequence of $y_j\geq \ell_j$. On the other hand, if $\ell_j\leq w_{i,j}$, then~\eqref{reduced-con} is equivalent to $y_j+(w_{i,j}-\ell_j)z_i\geq w_{i,j}$, and therefore we have $y_j\geq w_{i,j}$ when $z_i=0$ and have $y_j\geq \ell_j$ when $z_i=1$. Then, in both cases, it is clear that
	\[
	\left\{(y_j,z_i)\in\R\times\{0,1\}:~y_j\geq \ell_j,~y_j+(w_{i,j}-\ell_j)_+z_i\geq \ell_j+(w_{i,j}-\ell_j)_+\right\}
	\]
	is equal to
	\[
	\left\{(y_j,z_i)\in\R\times\{0,1\}:~y_j\geq \ell_j,~y_j+w_{i,j}z_i\geq w_{i,j}\right\},
	\]
	because $\ell_j\geq0$. 
	Hence, we have $(\vy-\vell,\vz)\in\cM({\vW}^{\vell},\bm{0},\varepsilon)$ if and only if $(\vy,\vz)\in \cM(\vW,\vell,\varepsilon)$, as required.
\end{proof}

We deduce from Lemma~\ref{lem:reduction} that
\begin{equation*}\label{reduction-hull}
\conv(\cM(\vW,\vell,\varepsilon))=\left\{(\vy,\vz)\in\R^k\times\R^n:~(\vy-\vell,\vz)\in\conv(\cM({\vW}^{\vell},\bm{0},\varepsilon))\right\},
\end{equation*}
and thus the convex hull description of $\cM(\vW,\vell,\varepsilon)$ can be obtained by taking the convex hull of $\cM({\vW}^{\vell},\bm{0},\varepsilon)$. Moreover, any inequality $\valpha^\top \vy+\vbeta^\top \vz\geq \gamma$ is valid for $\cM({\vW}^{\vell},\bm{0},\varepsilon)$ if and only if $\valpha^\top (\vy-\vell)+\vbeta^\top \vz\geq \gamma$ is valid for $\cM(\vW,\vell,\varepsilon)$. 

So, from now on, we assume that $\vell=\bm{0}$, and we work over~$\cM(\vW,\bm{0},\varepsilon)$ with $\vW\in\R^{n\times k}_+$ and $\varepsilon\geq0$. Recall  that $\cM(\vW,\bm{0},\varepsilon)$, which we call a joint mixing set with a linking constraint, is the mixed-integer set defined by 
\begin{subequations}
	\begin{align}
	&y_j+w_{i,j}z_i\geq w_{i,j},~~\quad \quad\forall i\in[n],~j\in[k],\label{bigM'}\\
	&y_j\geq 0,\quad\qquad\qquad\qquad\forall j\in[k],\label{lb'}\\
	&y_1+\cdots+y_k\geq\varepsilon,\label{linking'}\\
	&\vy\in\R^k,~\vz\in\{0,1\}^n.\label{binary'}
	\end{align}
\end{subequations}

Let us begin with an example.
\begin{example}\label{example1}
	Consider the following mixing set with a linking constraint, i.e., $\cM(\vW,\bm{0},\varepsilon)$ with $\varepsilon=7>0$.
	\begin{equation}\label{eq:example1}
	\left\{(\vy,\vz)\in\R_+^2\times\{0,1\}^5:~
	\begin{array}{l}
	y_1+ 8z_1\geq 8,\\
	y_1 + 6z_2\geq 6,\\
	y_1 + 13z_3\geq 13,\\
	y_1 + z_4\geq 1,\\
	y_1 + 4z_5\geq 4,
	\end{array}
	\begin{array}{l}
	y_2 + 3z_1 \geq 3,\\
	y_2 + 4z_2 \geq 4,\\
	y_2 + 2z_3 \geq 2,\\
	y_2 + 2z_4 \geq 2,\\
	y_2 + z_5 \geq 1,
	\end{array}
	~~y_1+y_2\geq 7\right\}.
	\end{equation}
	Using PORTA~\cite{porta}, we derive the convex hull description of this set, which is given by
		\[
		\left\{(\vy,\vz)\in\R_+^2\times[0,1]^5:~
		\begin{array}{l}
		\text{the mixing inequalities~\eqref{mixing-ineqs}},\\
		y_1+ y_2   +z_1+ z_2+ 8z_3 \geq 17,\\
		y_1+ y_2    +2z_2+ 8z_3 \geq 17,\\
		y_1+ y_2    +3z_2+ 7z_3 \geq 17,\\
		y_1+ y_2    +2z_1+3z_2+ 5z_3 \geq 17,\\
		y_1+ y_2+ 4z_1+z_2    + 5z_3 \geq 17
		\end{array}
		\right\}.
		\] In this example, the inequalities $y_1 +2z_1+2z_2+5z_3+z_4+3z_5\geq 13$ and $y_2+2z_2+z_4+z_5\geq 4$ are examples of mixing inequalities from~\eqref{mixing-ineqs} that are facet-defining. Note that the five inequalities with $y_1+y_2$ are not of the form~\eqref{mixing-ineqs}.   Moreover, these non-mixing inequalities  cannot be obtained by simply adding one mixing inequality involving $y_1$ and another mixing inequality involving $y_2$. %
	The developments we present next on a new class of inequalities will demonstrate this point, and we will revisit this example again in Example~\ref{example1revisit}.
\end{example}

The five inequalities with $y_1+y_2$ in Example~\ref{example1} admit a common interpretation. To explain them, take an integer $\theta\in[n]$ and a sequence $\Theta$ of $\theta$ indices in $[n]$ given by $\{i_1\to i_2\to\cdots\to i_\theta\}$. Given two indices in the sequence $i_p,i_q$, we say that {\it $i_p$ precedes $i_q$ in~$\Theta$} if $p< q$. 
Our description is based on the following definition.
\begin{definition}\label{def:mixing-subseq}
Given a sequence $\Theta$, 
a {\it $j$-mixing-subsequence of $\Theta$} is the subsequence $\{j_1\to\cdots\to j_{\tau_j}\}$ of $\Theta$ that satisfies the following property:
\[
\left\{j_1,\ldots,j_{\tau_j}\right\}~\text{is the collection of all indices $i^*\in\Theta$ satisfying}~w_{i^*,j}\geq\max\left\{w_{i,j}:~\text{$i^*$ precedes $i$  in $\Theta$}\right\},
\]
where we define $\max\left\{w_{i,j}:~\text{$i_\theta$ precedes $i$  in $\Theta$}\right\}=0$ for convention ($i_\theta$ is the last element, so it precedes no element in $\Theta$). 
\end{definition}
Based on Definition~\ref{def:mixing-subseq}, we deduce that the $j$-mixing-subsequence of $\Theta$ is unique for each $j\in[k]$ and admits a few nice structural properties as identified below.
\begin{lemma}\label{lem:j-sequence}
If $\{j_1\to \cdots \to j_{\tau_j}\}$ is the $j$-mixing-subsequence of $\Theta$, then $j_{\tau_j}$ is always the last element $i_\theta$ of $\Theta$ and $w_{j_1,j}\geq\cdots\geq w_{j_\tau, j}\geq0$.
\end{lemma}
\begin{proof}
When $p<q$, because $j_p$ precedes $j_q$ in $\Theta$, it follows that $w_{j_1,j}\geq\cdots\geq w_{j_{\tau_j},j}\geq0$.  The last element $i_\theta$ always satisfies $w_{i_\theta, j}\geq \max\left\{w_{i,j}:~\text{$i_\theta$ precedes $i$  in $\Theta$}\right\}=0$. Therefore, $i_\theta$ is part of the $j$-mixing-subsequence as its last element.
\end{proof}
Given  $\Theta=\{i_1\to i_2\to\cdots\to i_\theta\}$, for any $j\in[n]$, we denote by $\Theta_j=\{j_1\to\cdots\to j_{\tau_j}\}$ the $j$-mixing-subsequence of $\Theta$. By Definition~\ref{def:j-mixing-seq} and Lemma~\ref{lem:j-sequence}, we deduce that $\{j_1\to\cdots\to j_{\tau_j}\}$ is a $j$-mixing-sequence. Recall that for any $j$-mixing-sequence $\{j_1\to\cdots\to j_{\tau_j}\}$, the corresponding mixing inequality~\eqref{mixing-ineqs} is of the following form:
\begin{equation}\label{mixing}
y_j+\sum_{s\in[\tau_j]} (w_{j_s,j}-w_{j_{s+1},j})z_{j_s}\geq w_{j_1,j}\tag{Mix},
\end{equation}
where $w_{j_{\tau_j+1},j}:=0$, and it is valid for $\cM(\vW,\bm{0},\varepsilon)$. In particular, when $w_{j_1,j}=\max\{w_{i,j}:~i\in[n]\}$,~\eqref{mixing} is
\begin{equation}\label{mixing'}
y_j+\sum_{s\in[\tau_j]} (w_{j_s,j}-w_{j_{s+1},j})z_{j_s}\geq \max\{w_{i,j}:~i\in[n]\}\tag{Mix*}.
\end{equation}
Also, for $t\in[\theta]$,
\begin{equation}\label{coeff-transf}
\left(w_{i_t,j}   -  \max\left\{w_{i,j}:~\text{$i_t$ precedes $i$   in $\Theta$}\right\}   \right)_+=
\begin{cases} 
w_{j_{s},j}-w_{j_{s+1},j} &\text{if $i_t=j_{s}$ for some $s\in[\tau_j]$},\\
0 & \text{if $i_t$ is not on $\Theta_j$}.
\end{cases}
\end{equation}
Then~\eqref{mixing} can be rewritten as
\begin{equation*}\label{mixing-expanded}
y_j+\sum_{t\in[\theta]} \left(w_{i_t,j}   -  \max\left\{w_{i,j}:~\text{$i_t$ precedes $i$   in $\Theta$}\right\}   \right)_+z_{i_t}\geq w_{j_1,j}.
\end{equation*}
In order to introduce our new class of inequalities, we define a constant $L_{\vW,\Theta}$ that depends on $\vW$ and $\Theta$ as follows:
\begin{align}
L_{\vW,\Theta}&:=\min\left\{\sum_{j\in[k]}\left(w_{i_t,j}- \left(w_{i_t,j}   -  \max\left\{w_{i,j}:~\text{$i_t$ precedes $i$   in $\Theta$}\right\}\right)_+\right):~t\in[\theta]\right\}\label{L_WTheta}\\
&=\min\left\{\sum_{j\in[k]}\min\left\{w_{i_t,j},~\max\left\{w_{i,j}:~\text{$i_t$ precedes $i$   in $\Theta$}\right\}\right\}:~t\in[\theta]\right\}\nonumber
\end{align}
Now we are ready to introduce our new class of inequalities. 
\begin{definition}\label{def:aggMixingIneq}
Given a sequence $\Theta=\{i_1\to i_2\to\cdots\to i_\theta\}$, let $L_{\vW,\Theta}$ be defined as in \eqref{L_WTheta}. Then, 
the {\it aggregated mixing inequality derived from $\Theta$} is defined as the following:
\begin{equation}\label{aggregate}
\sum_{j\in[k]}\left(y_j+\sum_{s\in[\tau_j]} (w_{j_s,j}-w_{j_{s+1},j})z_{j_s}\right) - \min\left\{\varepsilon,~L_{\vW,\Theta}\right\} z_{i_\theta}\geq \sum_{j\in[k]} \max\left\{w_{i,j}:~i\in\Theta\right\}\tag{A-Mix}.
\end{equation}
\end{definition}
\begin{remark}\label{rem:aggMixingIneqDomination}
Since $\min\left\{\varepsilon,~L_{\vW,\Theta}\right\}\geq 0$, the aggregated mixing inequality~\eqref{aggregate} dominates what is obtained after adding up the mixing inequalities~\eqref{mixing} for $j\in[k]$. 
\end{remark}

Before proving validity of~\eqref{aggregate}, we present an example illustrating how the aggregated mixing inequalities are obtained.
\begin{example}\label{example1revisit}
	We revisit the mixed-integer set in Example~\ref{example1}. Now take a sequence $\Theta=\{2\to 1\to 3\}$. Then $\{3\}$ and $\{2\to1\to3\}$ are the $1$-mixing-subsequence and $2$-mixing-subsequence of $\Theta$, respectively. Moreover,
	\begin{align*}
	L_{\vW,\Theta}&=\min\left\{  (6-(6-13)_+) + (4- (4-3)_+),~~(8- (8-13)_+)+ (3- (3-2)_+),~~13+2 \right\}  \\
	&=\min\left\{6+3, 8+2 ,13+2 \right\}=9.
	\end{align*}
	In~\eqref{eq:example1}, we have $\varepsilon=7$. Since $\varepsilon\leq L_{\vW,\Theta}$, the corresponding \eqref{aggregate} is 
	\[
	\left(y_1+ 13 z_3\right) +\left(y_2+ (4-3)z_2 + (3-2)z_1+ 2z_3\right) - 7 z_3 \geq 13 + 4,
	\]
	that is $y_1+y_2+z_1+z_2+8z_3\geq 17$. In Example~\ref{example1}, the other four inequalities with $y_1+y_2$ are also of the form~\eqref{aggregate}, and they are derived from the sequences $\left\{2\to 3\right\}$, $\left\{3\to 2\right\}$, $\left\{3\to 1\to 2\right\}$, and $\left\{3\to 2\to 1\right\}$. So, in this example, the convex hull of~\eqref{eq:example1} is obtained after applying the mixing inequalities~\eqref{mixing} and the aggregated mixing inequalities~\eqref{aggregate}.
\end{example}

We will next present the proof of validity of ~\eqref{aggregate}. To this end, the following lemma is useful. As the proof of this lemma is technical, we defer its proof to the appendix. Lemma~\ref{lem:aggregate-lemma} will be used again when proving Theorem~\ref{thm:main}.

\begin{lemma}\label{lem:aggregate-lemma}
Let $(\bm{\bar y},\bm{\bar z})\in\R_+^k\times [0,1]^n$ be a point satisfying~\eqref{bigM'}--\eqref{linking'}. If $(\bm{\bar y},\bm{\bar z})$ satisfies~\eqref{aggregate} for all sequences contained in $\left\{i\in[n]:~\bar z_i<1\right\}$, then $(\bm{\bar y},\bm{\bar z})$ satisfies~\eqref{aggregate} for all the other sequences as well.
\end{lemma}

Now we are ready to prove the following theorem:

\begin{theorem}\label{thm:pm-aggregate}
The aggregated mixing inequalities defined as in~\eqref{aggregate} are valid for $\cM(\vW,\bm{0},\varepsilon)$ where $\vW\in\R^{n\times k}_+$.
\end{theorem}
\begin{proof}
We will argue that every point in $\cM(\vW,\bm{0},\varepsilon)$ with $\vW\in\R^{n\times k}_+$ satisfies~\eqref{aggregate} for all sequences in $[n]$. To this end, take a point $(\bm{\bar y},\bm{\bar z})\in\cM(\vW,\bm{0},\varepsilon)$. Then, $\bar \vz\in\{0,1\}^n$ holds by definition of $\cM(\vW,\bm{0},\varepsilon)$. If $\bm{\bar z}=\bm{1}$, then $(\bm{\bar y},\bm{\bar z})$ satisfies~\eqref{aggregate} if and only if $\sum_{j\in[k]}\bar y_j\geq \min\left\{\varepsilon,~L_{\vW,\Theta}\right\}$. Since $\sum_{j\in[k]}\bar y_j\geq \varepsilon$, it follows that $(\bm{\bar y},\bm{\bar z})$ satisfies~\eqref{aggregate}. Thus, we may assume that $\left\{i\in[n]:~\bar z_i<1\right\}=\left\{i\in[n]:~\bar z_i=0\right\}$ is nonempty. By Lemma~\ref{lem:aggregate-lemma}, it is sufficient to show that $(\bm{\bar y},\bm{\bar z})$ satisfies~\eqref{aggregate} for every sequence contained in the nonempty set $\left\{i\in[n]:~\bar z_i<1\right\}$. Take a nonempty sequence $\Theta=\left\{i_1\to\cdots\to i_\theta\right\}$ in $\left\{i\in[n]:~\bar z_i=0\right\}$. By our choice of $\Theta$, we have $\bar z_{i_\theta}=0$, so $(\bm{\bar y},\bm{\bar z})$ satisfies~\eqref{aggregate} if and only if
\[
\sum_{j\in[k]}\left(\bar y_j+\sum_{s\in[\tau_j]} (w_{j_s,j}-w_{j_{s+1},j})\bar z_{j_s}\right)\geq \sum_{j\in[k]} w_{j_1,j}.
\]
This inequality is precisely what is obtained by adding up the mixing inequalities~\eqref{mixing} for $j\in[k]$, and therefore, $(\bm{\bar y},\bm{\bar z})$ satisfies~it, as required.
\end{proof}

In Example~\ref{example1revisit}, $\varepsilon=7$ and $L_{\vW,\{2\to1\to3\}}=9$. It can also be readily checked that $L_{\vW,\{2\to3\}}=L_{\vW,\{3\to2\}}=8$ and $L_{\vW,\{3\to1\to2\}}=L_{\vW,\{3\to2\to1\}}=9$, which means $\min\left\{\varepsilon,L_{\vW,\Theta}\right\}=\varepsilon$ for the sequences corresponding to the five aggregated mixing inequalities in the convex hull description of~\eqref{eq:example1}. In general, the following holds:

\begin{proposition}
	If $\varepsilon\leq L_{\vW,\Theta}$, then the aggregated mixing inequality~\eqref{aggregate} obtained from $\Theta$ dominates the linking constraint $y_1+\cdots+y_k\geq \varepsilon$.
\end{proposition}
\begin{proof}
	The inequality~\eqref{aggregate} is equivalent to
	\[
	\sum_{j\in[k]}y_j\geq \varepsilon z_{i_\theta}+\sum_{j\in[k]} \left(w_{j_1,j} - \sum_{s\in[\tau_j]} (w_{j_s,j}-w_{j_{s+1},j})z_{j_s}\right).
	\]
	Since $\sum_{s\in[\tau_j]} (w_{j_s,j}-w_{j_{s+1},j})=w_{j_1,j}$, we deduce by by Lemma~\ref{lem:j-sequence} that for all $j\in[k]$,
\begin{align*}
w_{j_1,j} - \sum_{s\in[\tau_j]} (w_{j_s,j}-w_{j_{s+1},j})z_{j_s}
& = \sum_{s\in[\tau_j]} (w_{j_s,j}-w_{j_{s+1},j})(1-z_{j_s}) \\
&\geq (w_{j_{\tau_j},j}-w_{j_{{\tau_j}+1},j})(1-z_{j_{\tau_j}}) =w_{i_\theta, j}(1-z_{i_\theta}), 
\end{align*}
where the inequality follows from the facts that $w_{j_s,j}-w_{j_{s+1},j}\geq 0$ for all $j_s\in[\tau_j]$ and thus each summand is nonnegative, and the last equation follows from $j_{\tau_j}=i_{\theta}$ and by our convention that $w_{j_{\tau_j+1},j}=0$. 
Therefore, the following inequality is a consequence of~\eqref{aggregate}:
	\[
	\sum_{j\in[k]}y_j\geq \sum_{j\in[k]}w_{i_\theta, j}+\left(\varepsilon-\sum_{j\in[k]}w_{i_\theta, j}\right) z_{i_\theta}.
	\]
	Since $0\leq z_{i_\theta}\leq 1$, its right-hand side is always greater than or equal to $\min\left\{\sum_{j\in[k]}w_{i_\theta, j},~\varepsilon \right\}$. Since $\max\left\{w_{i,j}:~\text{$i_\theta$ precedes $i$   in $\Theta$}\right\}=0$, it follows from the definition of $L_{\vW,\Theta}$ in~\eqref{L_WTheta} that $\sum_{j\in[k]}w_{i_\theta, j}\geq L_{\vW,\theta}$. Then, by our assumption that $L_{\vW,\Theta}\geq\varepsilon$, we have $\min\left\{\sum_{j\in[k]}w_{i_\theta, j},~\varepsilon \right\}=\varepsilon$, %
	implying in turn that $y_1+\cdots+y_k\geq \varepsilon$ is implied by~\eqref{aggregate}, as required.
\end{proof}

We next demonstrate that when $\varepsilon$ is large, applying the aggregated mixing inequalities is not always enough to describe the convex hull of $\cM(\vW,\bm{0},\varepsilon)$ via an example.

\begin{example}\label{example2}
	The following set is the same as~\eqref{eq:example1} in Examples~\ref{example1}~and~\ref{example1revisit} except that $\varepsilon=9$. 
	\begin{equation}\label{eq:example2}
	\left\{(\vy,\vz)\in\R_+^2\times\{0,1\}^5:~
	\begin{array}{l}
	y_1+ 8z_1\geq 8,\\
	y_1 + 6z_2\geq 6,\\
	y_1 + 13z_3\geq 13,\\
	y_1 + z_4\geq 1,\\
	y_1 + 4z_5\geq 4,
	\end{array}
	\begin{array}{l}
	y_2 + 3z_1 \geq 3,\\
	y_2 + 4z_2 \geq 4,\\
	y_2 + 2z_3 \geq 2,\\
	y_2 + 2z_4 \geq 2,\\
	y_2 + z_5 \geq 1,
	\end{array}
	~~\bm{y_1+y_2\geq 9}\right\}.
	\end{equation}
	Recall that $L_{\vW,\{2\to 3\}}=8$, so $\varepsilon>L_{\vW,\{2\to3\}}$ in this case. As before, we obtain the convex hull description of~\eqref{eq:example2} via PORTA~\cite{porta}:
	\[
	\left\{(\vy,\vz)\in\R_+^2\times[0,1]^5:~
	\begin{array}{l}	
	\text{the mixing inequalities~\eqref{mixing}}\\
	7y_1+ 6y_2   + 12z_2+ 49z_3 \geq 115,\\
	6y_1+ 5y_2   + 10z_2+ 42z_3+z_4 \geq 98,\\
	3y_1+ 2y_2   + 4z_2+ 21z_3+z_4+3z_5 \geq 47,\\
	3y_1+ 2y_2   + 4z_2+ 21z_3+4z_5 \geq 47,\\
	2y_1+ 3y_2    +6z_2+ 14z_3 \geq 38,\\
	~~y_1+ 2y_2   + 4z_2+ 7z_3+ z_5 \geq 21,\\
	~~y_1+~~y_2    +z_1+z_2+ 6z_3 \geq 17,\\
	~~y_1+~~y_2+ 2z_1+z_2    + 5z_3 \geq 17\\
	\end{array}
	\right\}.
	\]
In this convex hull description, there are still two inequalities with $y_1+y_2$, and it turns out that these are aggregated mixing inequalities. To illustrate, take a sequence $\Theta=\{2\to 1\to 3\}$. We observed in Example~\ref{example1revisit} that $\{3\}$ and $\{2\to 1\to 3\}$ are the 1-mixing subsequence and the 2-mixing subsequence of $\Theta$ and that $L_{\vW,\Theta}=9$. So, the corresponding aggregated mixing inequality~\eqref{aggregate} is $y_1+y_2+z_1+z_2+6z_3\geq 17$. Similarly, we obtain $y_1+y_2+2z_1+z_2+5z_3\geq 17$ from $\{3\to 2\to 1\}$. However, unlike the system~\eqref{eq:example1} in Example~\ref{example1}, there are facet-defining inequalities for the convex hull of this set other than the aggregated mixing inequalities, i.e., the first 6 inequalities in the above description of the convex hull have different coefficient structures on the $y$ variables.
\end{example}

So, a natural question is: When are the mixing inequalities and the aggregated mixing inequalities  sufficient to describe the convex hull of $\cM(\vW,\bm{0},\varepsilon)$? Examples~\ref{example1}--\ref{example2} suggest that whether or not the mixing and the aggregated mixing inequalities are sufficient depends on the value of $\varepsilon$. In the next section, we find a necessary and sufficient condition for the sufficiency of the mixing and the aggregated mixing inequalities.

\section{Joint mixing sets with a linking constraint}\label{sec:mixing-lb}

In this section, we study the convex hull of $\cM(\vW,\bm{0},\varepsilon)$, %
where $\vW=\left\{w_{i,j}\right\}\in\R_+^{n\times k}$ and $\varepsilon\in\R_+$. More specifically, we focus on the question of when the convex hull of this set is obtained after applying the mixing inequalities and the aggregated mixing inequalities. 
By Remark~\ref{rem:M-P-relation}, we have $(\vy,\vz)\in\cM(\vW,\bm{0},\varepsilon)$ if and only if $(\vy,\bm{1}-\vz)\in\cP(\vW,\bm{0},\varepsilon)$. 
In Section~\ref{sec:mixing}, we identified that $\cP(\vW,\vell,0)$ defined as in~\eqref{twisted-mixingset} has an underlying submodularity structure (due to Lemma~\ref{lem:submodular1} and Proposition~\ref{prop:intersection1}). In this section, we will first establish %
that $\cP(\vW,\bm{0},\varepsilon)$ has a similar submodularity structure for particular values of $\varepsilon$. In fact, for those favorable values of $\varepsilon$, we show that the mixing and the aggregated mixing inequalities are sufficient to describe the convex hull of $\cM(\vW,\bm{0},\varepsilon)$ if and only if $\cP(\vW,\bm{0},\varepsilon)$ has the desired submodularity structure; this is the main result of this section.

\subsection{Submodularity in joint mixing sets with a linking constraint}\label{sec:submod}

In order to make a connection with submodularity, we first define the following functions $f_1,\ldots,f_k,g:\{0,1\}^n\to\R$: for $\vz\in\{0,1\}^n$,
\begin{equation}\label{submod}
f_j(\vz):=\max\limits_{i\in[n]}\left\{w_{i,j}z_i\right\}\quad \text{for}~j\in[k]\quad\text{and}\quad g(\vz):=
\max\left\{\varepsilon, \sum_{j\in[k]}f_j(\vz)\right\} .
\end{equation}
Then, we immediately arrive at the following representation of $\cP(\vW,\bm{0},\varepsilon)$. 
\begin{lemma}\label{lem:P-epigraph}
Let $f_1,\ldots,f_k,g:\{0,1\}^n\to\R$ be as defined in \eqref{submod}. Then, 
\begin{equation}\label{P-submod}
\cP(\vW,\bm{0},\varepsilon)=\left\{(\vy,\vz)\in\R^k\times\{0,1\}^n:~y_j\geq f_j(\vz),~\forall j\in[k],~~ y_1+\cdots+y_k\geq g(\vz)\right\}.
\end{equation} 
\end{lemma}
\begin{proof}
We deduce the equivalence of the relations $y_j\geq f_j(\vz)$ for $j\in[k]$ to the first set of constraints in $\cP(\vW,\bm{0},\varepsilon)$ from the corresponding definition of this set in~\eqref{twisted-mixingset}. Also, we immediately have $\sum_{j\in[k]}y_j\geq \max\left\{ \varepsilon,\, \sum_{j\in[k]}f_j(\vz) \right\}$. The result then follows from the definition of the function $g$.
\end{proof}

We would like to understand the convex hull of $\cP(\vW,\bm{0},\varepsilon)$ for $\vW\in\R^{n\times k}_+$ and $\varepsilon\in\R_+$ using  Lemma~\ref{lem:P-epigraph}. Observe that $f_1,\ldots,f_k$ defined in~\eqref{submod} coincide with the functions $f_1,\ldots,f_k$ defined in~\eqref{y_j-submod} for the $\vell=\bm{0}$ case. So, the following is a direct corollary of Lemma~\ref{lem:submodular1}.
\begin{corollary}\label{cor:f-submodular}
For any $j\in[k]$, the function $f_j$ defined as in~\eqref{submod} is submodular and satisfies $f_j(\emptyset)\geq0$.  
\end{corollary}

In contrast to the functions $f_1,\ldots,f_k$, the function $g$ is not always submodular. However, we can characterize exactly when $g$ is submodular in terms of $\varepsilon$. For this characterization, we need to define several parameters based on $\vW$ and $\varepsilon$. For a given $\varepsilon$, let $\bar I(\varepsilon)$ be the following subset of $[n]$: 
\begin{equation}\label{indexset}
\bar I(\varepsilon):=\left\{i\in[n]:~\sum_{j\in[k]}w_{i,j}\leq\varepsilon\right\}.
\end{equation}
Here, $\bar I(\varepsilon)$ is the collection of indices $i$ with $g(\{i\})=\varepsilon$. In Examples~\ref{example1} and~\ref{example2}, we have $\bar I(\varepsilon)=\{4,5\}$. 
\begin{definition}\label{def:eps-negligible}
	We say that $\bar I(\varepsilon)$ is {\it $\varepsilon$-negligible} if either 
	\begin{itemize}
		\item $\bar I(\varepsilon)=\emptyset$ or 
		\item $\bar I(\varepsilon)\neq\emptyset$ and $\bar I(\varepsilon)$ satisfies both of the following two conditions:
		\begin{align}
		&\text{$\quad\max\limits_{i\in \bar I(\varepsilon)}\left\{w_{i,j}\right\}\leq w_{i,j}$ for every $i\in[n]\setminus \bar I(\varepsilon)$ and $j\in[k]$,}\tag{C1}\label{condition1}\\
		&\text{$\quad\sum_{j\in[k]}\max\limits_{i\in \bar I(\varepsilon)}\left\{w_{i,j}\right\}\leq \varepsilon$.}\tag{C2}\label{condition2}
		\end{align}
	\end{itemize}
\end{definition}
\begin{example}\label{ex1,2-negligible}
In Example~\ref{example1}, it can be readily checked that $\bar I(\varepsilon)$ satisfies~\eqref{condition1} and~\eqref{condition2}, so $I(\varepsilon)$ is $\varepsilon$-negligible. The matrix $\vW$ of Example~\ref{example2} is the same as that of Example~\ref{example1}, while the value of $\varepsilon$ is higher in Example~\ref{example2}. Hence, $\bar I(\varepsilon)$ in Example~\ref{example2} is also $\varepsilon$-negligible.
\end{example}
In Definition~\ref{def:eps-negligible}, \eqref{condition2} imposes that $g(\bar I(\varepsilon))=\varepsilon$, and~\eqref{condition1} requires that $f_j(\{i\}\cup \bar I(\varepsilon))=f_j(\{i\})$ for any $i\in[n]\setminus \bar I(\varepsilon)$. In fact, we can argue that if $\bar I(\varepsilon)$ is $\varepsilon$-negligible, $\bar I(\varepsilon)$ does not affect the value of $g$; this is why we call this property $\varepsilon$-``negligibility." The following lemma formalizes this observation.
\begin{lemma}\label{lem:negligible}
Let $g$ be as defined in \eqref{submod}. If $\bar I(\varepsilon)$ is $\varepsilon$-negligible, then $g(U)=g(U\setminus \bar I(\varepsilon))$ for every $U\subseteq [n]$.
\end{lemma}
\begin{proof}
Suppose $\bar I(\varepsilon)$ is nonempty and satisfies conditions~\eqref{condition1} and~\eqref{condition2}. Take a subset $U$ of $[n]$. If $U\subseteq \bar I(\varepsilon)$, then $g(U)\leq g(\bar I(\varepsilon))$ because $g$ is a monotone function. Since $\sum_{j\in[k]}\max\limits_{i\in \bar I(\varepsilon)}\left\{w_{i,j}\right\}\leq \varepsilon$, we obtain $g(\bar I(\varepsilon))=\varepsilon$ by definition of $g$ in~\eqref{submod}. So, $g(U)=g(\emptyset)=\varepsilon$ in this case. If $U\setminus \bar I(\varepsilon)\neq\emptyset$, then $\sum_{j\in[k]}w_{p,j}>\varepsilon$ for some $p\in U$, implying in turn that $\sum_{j\in [k]}\max_{i\in U}\left\{w_{i,j}\right\}>\varepsilon$. Moreover, as $\bar I(\varepsilon)$ satisfies~\eqref{condition1}, $\sum_{j\in [k]}\max_{i\in U}\left\{w_{i,j}\right\}=\sum_{j\in [k]}\max_{i\in U\setminus \bar I(\varepsilon)}\left\{w_{i,j}\right\}$, and therefore, $g(U)=g(U\setminus \bar I(\varepsilon))$, as required.
\end{proof}
Next we show that $\varepsilon$-negligibility is necessary for $g$ to be submodular.
\begin{lemma}\label{g-submod-negligible}
If the function $g$ defined as in~\eqref{submod} is submodular,	then $\bar I(\varepsilon)$ is $\varepsilon$-negligible.
\end{lemma}
\begin{proof}
Assume that $g$ is submodular. Suppose for a contradiction that $\bar I(\varepsilon)$ is not $\varepsilon$-negligible. Then $\bar I(\varepsilon)$ is nonempty, and~\eqref{condition1} or~\eqref{condition2} is violated. Assume that $\bar I(\varepsilon)$ does not satisfy~\eqref{condition1}. Then $w_{q,j}>w_{p,j}$ for some $j\in[k]$, $p\in [n]\setminus \bar I(\varepsilon)$ and $q\in \bar I(\varepsilon)$. By our choice of $q$, we have $g(\left\{q\right\})=\varepsilon$. Moreover, $w_{q,j}>w_{p,j}$ implies that $g(\left\{p,q\right\})=\sum_{j\in[k]}\max\{w_{p,j},w_{q,j}\}>\sum_{j\in[k]}w_{p,j}=g(\left\{p\right\})$. Since $g(\emptyset)=\varepsilon$, it follows that $g(\{p\})+g(\{q\})<g(\emptyset)+g(\{p,q\})$, a contradiction to the submodularity of $g$. Thus, we may assume that $\bar I(\varepsilon)$ does not satisfy~\eqref{condition2}. Then $\sum_{j\in[k]}\max\limits_{i\in \bar I(\varepsilon)}\left\{w_{i,j}\right\}>\varepsilon$, so $g(\bar I(\varepsilon))=\sum_{j\in[k]}\max\limits_{i\in \bar I(\varepsilon)}\left\{w_{i,j}\right\}$. Now take a minimal subset $I$ of $\bar I(\varepsilon)$ with $g(I)>\varepsilon$. Since $I\subseteq \bar I(\varepsilon)$ and $g(I)>\varepsilon$, we know that $|I|\geq 2$. That means that one can find two nonempty subsets $U,V$ of $I$ partitioning $I$. By our minimal choice of $I$, we have $g(U)=g(V)=\varepsilon$, but this indicates that $g(U)+g(V)<g(\emptyset)+g(I)=g(U\cap V)+g(U\cup V)$, a contradiction to the submodularity of $g$. Therefore, $\bar I(\varepsilon)$ is $\varepsilon$-negligible.
\end{proof}

On the other hand, it turns out that $\varepsilon$-negligibility alone does not always guarantee that $g$ is submodular. If $\bar I(\varepsilon)=[n]$, $\bar I(\varepsilon)$ being $\varepsilon$-negligible means that $g(U)=g(\emptyset)=\varepsilon$ for every $U\subseteq[n]$ and thus $g$ is clearly submodular. However, when $\bar I(\varepsilon)$ is a strict subset of $[n]$, $g$ may not necessarily be submodular even though $\bar I(\varepsilon)$ is $\varepsilon$-negligible.
\begin{example}\label{g-not-submodular}
In Example~\ref{example2}, we have observed that $\bar I(\varepsilon)=\{4,5\}$ and $\bar I(\varepsilon)$ is $\varepsilon$-negligible. By definition, we have $g(\emptyset)=\varepsilon=9$. Since $2,3\notin \bar I(\varepsilon)$, we have that $g(\{2\})=w_{2,1} +w_{2,2}=10$, $g(\{3\})=w_{3,1}+w_{3,2}=15$, and $g(\{2,3\})=\max\{w_{2,1},w_{3,1}\}+\max\{w_{2,2}+w_{3,2}\}=17$. Then $g(\{2\})+g(\{3\})=25$ is less than $g(\{2,3\})+g(\emptyset)=26$, so $g$ is not submodular. 
\end{example}
In order to understand when the function $g$ is submodular, let us take a closer look at Example~\ref{g-not-submodular}. In this example, $g(\{2\})+g(\{3\})-g(\{2,3\})$ is equal to $\min\{w_{2,1},w_{3,1}\}+\min\{w_{2,2}+w_{3,2}\}$, and this value is less than $\varepsilon=g(\emptyset)$, implying that $g$ is not submodular. In general, for any distinct indices $p,q\in[n]\setminus\bar I(\varepsilon)$, 
\begin{equation}\label{eq:p,q-submod}
g(\{p\})+g(\{q\})-g(\{p,q\})=\sum\limits_{j\in[k]}\min\left\{w_{p,j},~w_{q,j}\right\},
\end{equation}
and this quantity needs to be greater than or equal to $\varepsilon=g(\emptyset)$ for $g$ to be submodular. To formalize this, we define another parameter $L_{\vW}(\varepsilon)\in\R_+$ as follows:
\begin{equation}\label{L_W}
L_{\vW}(\varepsilon):=\begin{cases}
\min\limits_{p,q\in [n] \setminus \bar I(\varepsilon)}\left\{\sum\limits_{j\in[k]}\min\left\{w_{p,j},~w_{q,j}\right\}\right\},&\text{if}~~\bar I(\varepsilon)\neq[n],\\
+\infty,&\text{if}~~\bar I(\varepsilon)=[n].
\end{cases}
\end{equation}
\begin{example}\label{ex1,2-parameters}
	In Example~\ref{example1}, we have $\bar I(\varepsilon)=\{4,5\}$ and $L_{\vW}(\varepsilon)=w_{2,1}+w_{3,2}=8$. Moreover, as $\bar I(\varepsilon)=\{4,5\}$ in Example~\ref{example2} as well, we still have $L_{\vW}(\varepsilon)=8$ in Example~\ref{example2}. 
\end{example}
\begin{lemma}\label{g-submod-LW}
	If the function $g$ defined as in~\eqref{submod} is submodular, then $\varepsilon\leq L_{\vW}(\varepsilon)$.
\end{lemma}
\begin{proof}
Suppose for a contradiction that $\varepsilon>L_{\vW}(\varepsilon)$. Then, $L_{\vW}(\varepsilon)\neq\infty$, implying  $\bar I(\varepsilon)\neq [n]$ and $\varepsilon>\sum_{j\in[k]}\min\left\{w_{p,j},~w_{q,j}\right\}$ for some $p,q\in [n] \setminus \bar I(\varepsilon)$. Moreover, because both $\sum_{j\in[k]} w_{p,j}$ and $\sum_{j\in[k]}w_{q,j}$ are greater than $\varepsilon$, we deduce that $p$ and $q$ are distinct. Then,  
	\begin{align*}
	g\left(\{p\}\right)+g\left(\{q\}\right)=\sum_{j\in[k]} w_{p,j}+\sum_{j\in[k]}w_{q,j}&=\sum_{j\in[k]}\max\left\{w_{p,j},~w_{q,j}\right\}+\sum_{j\in[k]}\min\left\{w_{p,j},~w_{q,j}\right\}\\
	&=g\left(\{p,q\}\right)+\sum_{j\in[k]}\min\left\{w_{p,j},~w_{q,j}\right\}<g\left(\{p,q\}\right)+g\left(\emptyset\right),
	\end{align*}
	where the strict inequality follows from $g(\emptyset)=\varepsilon$. This is a contradiction to the assumption that $g$ is submodular. Hence, $\varepsilon\leq L_{\vW}(\varepsilon)$, as required.
\end{proof}
By Lemmas~\ref{g-submod-negligible} and~\ref{g-submod-LW}, both of the conditions that $\bar I(\varepsilon)$ is $\varepsilon$-negligible and $\varepsilon\leq L_{\vW}(\varepsilon)$ are necessary for the submodularity of $g$. In fact, we will next see that these two conditions are also sufficient to guarantee that $g$ is submodular. So, whether the function $g$ is submodular or not is determined entirely by $\bar I(\varepsilon)$ and $L_{\vW}(\varepsilon)$.

\begin{lemma}\label{lem:g-submodular}
The function $g$ defined as in~\eqref{submod} is submodular if and only if $\bar I(\varepsilon)$ is $\varepsilon$-negligible and $\varepsilon\leq L_{\vW}(\varepsilon)$.
\end{lemma}
\begin{proof}
{\bf ($\Rightarrow$)}: This direction is settled by Lemmas~\ref{g-submod-negligible} and~\ref{g-submod-LW}.

{\bf ($\Leftarrow$)}: Assume that $\bar I(\varepsilon)$ is $\varepsilon$-negligible and $\varepsilon\leq L_{\vW}(\varepsilon)$. We will show that $g(U)+g(V)\geq g(U\cup V)+g(U\cap V)$ for every two sets $U,V\subseteq[n]$. If $\bar I(\varepsilon)=[n]$, then we have $g(U)=\varepsilon$ for every subset $U$ of $[n]$ due to~\eqref{condition2}. Thus, we may assume that $\bar I(\varepsilon)\neq[n]$. By Lemma~\ref{lem:negligible}, for every two subsets $U,V\subseteq[n]$, $g(U)+g(V)\geq g(U\cup V)+g(U\cap V)$ holds if and only if $g\left(U^\prime\right)+g\left(V^\prime\right)\geq g\left(U^\prime\cup V^\prime\right)+g\left(U^\prime\cap V^\prime\right)$,
where $U^\prime:=U\setminus \bar I(\varepsilon)$ and $V^\prime:=V\setminus \bar I(\varepsilon)$, holds. This means that it is sufficient to consider subsets of $[n]\setminus \bar I(\varepsilon)$. Consider two sets $U,V\subseteq [n]\setminus \bar I(\varepsilon)$. If $U=\emptyset$ or $V=\emptyset$, the inequality trivially holds due to the monotonicity of $g$. So, we may assume that $U,V\neq\emptyset$. First, suppose that $U\cap V\neq\emptyset$. Because $U,V\subseteq [n]\setminus \bar I(\varepsilon)$, %
we deduce that $g(X)=\sum_{j\in[k]}f_j(X)$ for any $X\in\{U,V,U\cup V,U\cap V\}$. Then, Corollary~\ref{cor:f-submodular} implies that $g(U)+g(V)\geq g(U\cup V)+g(U\cap V)$. Now, consider the case of $U\cap V=\emptyset$. Note that for each $j\in[k]$, the definition of $f_j(V)=\max_{i\in V}\{w_{i,j} \}$ implies that 
\[
	f_j(U)+f_j(V)-f_j(U\cup V) = \max\{f_j(U),f_j(V)\} + \min\{f_j(U),f_j(V)\} - f_j(U\cup V) = \min\{f_j(U),f_j(V)\}.
\]
Hence, we have
\[	
	g(U)+g(V)-g(U\cup V)=\sum_{j\in[k]}\left(f_j(U)+f_j(V)-f_j(U\cup V)  \right)=\sum_{j\in[k]} \min\{f_j(U),f_j(V)\}.
\]
So, it suffices to argue that $\sum_{j\in[k]} \min\{f_j(U),f_j(V)\}\geq g(\emptyset)=\varepsilon$. Since $U,V\neq\emptyset$ and $U\cap V=\emptyset$, there exist distinct $p,q\in [n]\setminus \bar I(\varepsilon)$ such that $p\in U$ and $q\in V$. Then $f_j(U)\geq f_j(\{p\})=w_{p,j}$ and $f_j(V)\geq f_j(\{q\})=w_{q,j}$, implying in turn that 
\[\sum_{j\in[k]} \min\{f_j(U),f_j(V)\}\geq \sum_{j\in[k]} \min\{w_{p,j},w_{q,j}\}\geq L_{\vW}(\varepsilon),
\]
where the last inequality follows from the definition of $L_{\vW}(\varepsilon)$ in~\eqref{L_W}. Finally, our assumption that $\varepsilon\leq L_{\vW}(\varepsilon)$ implies that $\sum_{j\in[k]} \min\{f_j(U),f_j(V)\}\geq \varepsilon$ as desired.
\end{proof}

Therefore, Lemma~\ref{lem:g-submodular}, along with Corollary~\ref{cor:f-submodular}, establish that $f_1,\ldots,f_k$ and $g$ are submodular when $\bar I(\varepsilon)$ is $\varepsilon$-negligible and $\varepsilon\leq L_{\vW}(\varepsilon)$. 
Note that $\bar I(\varepsilon)$ can be found in $O(kn)$ time and that $L_{\vW}(\varepsilon)$ can be computed in $O(kn^2)$ time, so testing whether $g$ is submodular can be done in polynomial time.

In Section~\ref{sec:aggregated}, we introduced the parameter $L_{\vW,\Theta}$ that depends on $\vW$ and a sequence $\Theta$ of indices in $[n]$ to define the aggregated mixing inequality~\eqref{aggregate} derived from $\Theta$. The following lemma illustrates a relationship between $L_{\vW}(\varepsilon)$ and $L_{\vW,\Theta}$:

\begin{lemma}\label{lem:L_Wbound}
	If $\bar I(\varepsilon)\neq[n]$, then $L_{\vW}(\varepsilon)=\min_{\Theta}\left\{L_{\vW,\Theta}:~\Theta~\text{is a nonempty sequence in}~[n]\setminus \bar I(\varepsilon)\right\}$.
\end{lemma}
\begin{proof}
	Take a nonempty sequence $\Theta$ in $[n]\setminus \bar I(\varepsilon)$. When $\Theta=\{r\}$ for some $r\in[n]\setminus \bar I(\varepsilon)$, $L_{\vW,\Theta}=\sum_{j\in[k]}w_{r,j}$, so $L_{\vW}(\varepsilon)\leq L_{\vW,\Theta}$ in this case. When $\Theta=\{i_1\to\cdots\to i_\theta\}$ with $\theta\geq 2$, for any $s\in[\theta]$ we have 
	\[\min\left\{w_{i_s,j},~\max\left\{w_{i,j}:~\text{$i_s$ precedes $i$  in $\Theta$}\right\}\right\}\geq \min\left\{w_{i_s,j},~w_{i_{s+1},j}\right\}\] 
	where $w_{i_{\theta+1},j}$ is set to 0 for convention. Then it follows from the definition of $L_{\vW,\Theta}$ in~\eqref{L_WTheta} that $L_{\vW,\Theta}\geq \min\left\{\sum_{j\in[k]}\min\left\{w_{i_s,j},~w_{i_{s+1},j}\right\}:~s\in[\theta]\right\}$. Consequently, from the definition of $L_{\vW}(\varepsilon)$, we deduce that 
	$L_{\vW}(\varepsilon)\leq L_{\vW,\Theta}$. In both cases, we get $L_{\vW}(\varepsilon)\leq L_{\vW,\Theta}$.
	
	Now it remains to show $L_{\vW}(\varepsilon)\geq\min\left\{L_{\vW,\Theta}:~\Theta~\text{is a nonempty sequence in}~[n]\setminus \bar I(\varepsilon)\right\}$. Since $\bar I(\varepsilon)\neq[n]$, either there exist distinct $p,q\in[n]\setminus \bar I(\varepsilon)$ such that $L_{\vW}(\varepsilon)=\sum_{j\in[k]}\min\left\{w_{p,j},~w_{q,j}\right\}=L_{\vW,\{p\to q\}}$ or there exists $r\in[n]\setminus \bar I(\varepsilon)$ such that $L_{\vW}(\varepsilon)=\sum_{j\in[k]}w_{r,j}=L_{\vW,\{r\}}$, implying in turn that $L_{\vW}(\varepsilon)\geq L_{\vW,\Theta}$ for some nonempty sequence $\Theta$ in $[n]\setminus \bar I(\varepsilon)$, as required.
\end{proof}

We have shown in Section~\ref{sec:mixing} that the polymatroid inequalities corresponding to the functions $f_1,\ldots,f_k$ are mixing inequalities. Although $g$ is not always submodular, we now have a complete characterization of when $g$ is submodular. In the next subsection, we show that when $g$ is indeed submodular, the corresponding polymatroid inequalities are aggregated mixing inequalities.

\subsection{Polymatroid inequalities and aggregated mixing inequalities}\label{sec:poly-agg}
Consider $\cP(\vW,\bm{0},\varepsilon)$ with $\vW\in\R_+^{n\times k}$ and $\varepsilon\in\R_+$. Then, from Lemma~\ref{lem:P-epigraph} we deduce that 
\[
\conv(\cP(\vW,\bm{0},\varepsilon))\subseteq\left\{(\vy,\vz)\in\R^k\times[0,1]^n:~(y_j,\vz)\in\conv(Q_{f_j}),~\forall j\in[k],~~ \left(y_1+\cdots+y_k,\vz\right)\in\conv(Q_{g})\right\},
\]
where $f_j,g$ are as defined in~\eqref{submod}. 
In this section we will prove that in fact equality holds in the above relation when $g$  is submodular, i.e., by Lemma~\ref{lem:g-submodular}, when $\bar I(\varepsilon)$ is $\varepsilon$-negligible and $\varepsilon\leq L_{\vW}(\varepsilon)$. 
Then, consequently, if $\bar I(\varepsilon)$ is $\varepsilon$-negligible and $\varepsilon\leq L_{\vW}(\varepsilon)$, then the separation problem over $\conv(\cP(\vW,\bm{0},\varepsilon))$ (equivalently,  $\conv(\cM(\vW,\bm{0},\varepsilon))$) can be solved in $O(kn\log n)$ time by a simple greedy algorithm. To this end, we first characterize the $\cV$-polyhedral, or inner, description of $\conv(\cP(\vW,\bm{0},\varepsilon))$. For notational purposes, we define a specific set of binary solutions as follows:
\begin{equation}\label{S_eps}
S(\varepsilon):=\left\{\vz\in\{0,1\}^n:~\sum_{j\in[k]}\max\limits_{i\in[n]}\left\{w_{i,j}z_i\right\}>\varepsilon\right\}.
\end{equation}
\begin{lemma}\label{lem:inner}
	The extreme rays of $\conv(\cP(\vW,\bm{0},\varepsilon))$ are $(\bm{e^j},\bm{0})$ for $j\in[k]$, and the extreme points are precisely the following:
		\begin{itemize}
			\item $A(z)=(\bm{y^z},\vz)$ for $\vz\in S(\varepsilon)$ where $y^z_j=\max\limits_{i\in[n]}\left\{w_{i,j}z_i\right\}$ for $j\in[k]$,
			\item $B(\vz,d)=(\bm{y^{z,d}},\vz)$ for $\vz\in\{0,1\}^n\setminus S(\varepsilon)$ and $d\in[k]$ where
			\[
			y^{z,d}_j=\begin{cases}
			\max\limits_{i\in[n]}\left\{w_{i,j}z_i\right\},&\text{if}~~j\neq d,\\
			\max\limits_{i\in[n]}\left\{w_{i,d}z_i\right\}+\left(\varepsilon-\sum_{j\in[k]}\max\limits_{i\in[n]}\left\{w_{i,j}z_i\right\}\right),&\text{if}~~j=d.
			\end{cases}
			\]
		\end{itemize}
\end{lemma}
\begin{proof}
It is clear that $(\bm{e^j},\bm{0})$ for $j\in[k]$ are the extreme rays of $\conv(\cP(\vW,\bm{0},\varepsilon))$. Let $(\bm{\bar y},\bm{\bar z})$ be an extreme point of $\conv(\cP(\vW,\bm{0},\varepsilon))$. Then $\bm{\bar z}\in\{0,1\}^n$, and constraints~\eqref{bigM''} become $\bar y_j\geq  \max\limits_{i\in[n]}\left\{w_{i,j}\bar z_i\right\}$ for $j\in[k]$. If $\bm{\bar z}\in S(\varepsilon)$, then $\sum_{j\in [k]}\max\limits_{i\in[n]}\left\{w_{i,j}\bar z_i\right\}>\varepsilon$, so $(\bm{\bar y},\bm{\bar z})$ automatically satisfies~\eqref{lb''}--\eqref{linking''}. As $(\bm{\bar y},\bm{\bar z})$ is an extreme point, it follows that $\bar y_j=\max\limits_{i\in[n]}\left\{w_{i,j}\bar z_i\right\}$ for $j\in[k]$, and therefore, $(\bm{\bar y},\bm{\bar z})=A(\bm{\bar z})$. If $\bm{\bar z}\notin S(\varepsilon)$, then $\sum_{j\in [k]}\max\limits_{i\in[n]}\left\{w_{i,j}\bar z_i\right\}\leq\varepsilon$. Since $(\bm{\bar y},\bm{\bar z})$ satisfies $\bar y_1+\cdots+\bar y_k\geq\varepsilon$ and $(\bm{\bar y},\bm{\bar z})$ cannot be expressed as a convex combination of two distinct points, it follows that $\bar y_1+\cdots+\bar y_k\geq\varepsilon$ and  constraints $\bar y_j\geq  \max\limits_{i\in[n]}\left\{w_{i,j}\bar z_i\right\},$  $j\in[k]\setminus \{d\}$ are tight at $(\bm{\bar y},\bm{\bar z})$ for some $d\in[k]$, so $(\bm{\bar y},\bm{\bar z})=B(\vz,d)$.
\end{proof}

Based on the definition of $S(\varepsilon)$ and \eqref{submod}, we have
\[
g(\vz)=\max\left\{\varepsilon, \sum_{j\in[k]}f_j(\vz)\right\} = 
\begin{cases}
\sum_{j\in[k]}f_j(\vz),&\text{if }\vz\in S(\varepsilon)\\
\varepsilon,&\text{if }\vz\not\in S(\varepsilon).
\end{cases}
\]
Remember the definition of $\bar I(\varepsilon)$ in~\eqref{indexset} and the conditions for $\bar I(\varepsilon)$ to be $\varepsilon$-negligible. Recall the definition of $L_{\vW}(\varepsilon)$ in~\eqref{L_W} as well. 
Based on these definitions and Proposition~\ref{prop:intersection2}, we are now ready to give the explicit inequality characterization of the convex hull of $\cM(\vW,\bm{0},\varepsilon)$.
\begin{proposition}\label{prop:mixing-polymatroid}
Let $\vW=\{w_{i,j}\}\in\R_+^{n\times k}$ and $\varepsilon\in\R_+$. If $\bar I(\varepsilon)$ is $\varepsilon$-negligible and $\varepsilon\leq L_{\vW}(\varepsilon)$, then the convex hull of $\cM(\vW,\bm{0},\varepsilon)$ is given by
\[
\left\{(\vy,\vz)\in\R^k\times[0,1]^n:~(y_j,\bm{1}-\vz)\in\conv(Q_{f_j}),~\forall j\in[k],~~ \left(y_1+\cdots+y_k,\bm{1}-\vz\right)\in\conv(Q_{g})\right\}.
\]
\end{proposition}
\begin{proof}
We will show that $y_1,\ldots,y_k$ and $\sum_{j\in[k]}y_j$ are weakly independent with respect to submodular functions $f_1,\ldots,f_k$ and $g$ (recall Definition \ref{def:weakindep}). Consider $\valpha\in\R_+^k\setminus\{\bm{0}\}$, and let $\alpha_{\min}$ denote the smallest coordinate value of $\valpha$. Then $\valpha$ and $\valpha^\top \vy$ can be written as $\valpha=\alpha_{\min} \bm{1}+\sum_{j\in[k]}(\alpha_j-\alpha_{\min})\bm{e^j}$ and $\valpha^\top \vy=\alpha_{\min} \sum_{j\in[k]}y_j+\sum_{j\in[k]}(\alpha_j-\alpha_{\min})y_j$.
Let $f_{\valpha}$ be defined as $f_{\valpha}(\vz):=\min\left\{\valpha^\top \vy:~(\vy,\vz)\in \cP(\vW,\bm{0},\varepsilon)\right\}$ for $\vz\in\{0,1\}^n$. Then, it is sufficient to show that $f_{\valpha}=\alpha_{\min} g+\sum_{j\in[k]}(\alpha_j-\alpha_{\min})f_j$.

Let $\bm{\bar z}\in\{0,1\}^n$. For any $\vy$ with $(\vy,\bm{\bar z})\in \cP(\vW,\bm{0},\varepsilon)$, we have $y_j\geq f_j(\bm{\bar z})$ for $j\in[k]$ and $\sum_{j\in[k]}y_j\geq g(\bm{\bar z})$ by Lemma~\ref{lem:P-epigraph}, implying in turn that 
\begin{equation}\label{f_alpha-lb}
f_{\valpha}(\bm{\bar z})=\min\left\{\valpha^\top \vy:(\vy,\bm{\bar z})\in \cP(\vW,\bm{0},\varepsilon)\right\}\geq \alpha_{\min} g(\bm{\bar z})+\sum_{j\in[k]}(\alpha_j-\alpha_{\min})f_j(\bm{\bar z}).
\end{equation}
Recall the definition of $S(\varepsilon)$ in~\eqref{S_eps}. If $\bm{\bar z}\in S(\varepsilon)$, then $g(\bm{\bar z})=\sum_{j\in[k]}f_j(\bm{\bar z})$, and therefore, $A(\bm{\bar z})=(\bm{y^{\bar z}},\bm{\bar z})$ defined in Lemma~\ref{lem:inner} satisfies~\eqref{f_alpha-lb} at equality. If $\bm{\bar z}\notin S(\varepsilon)$, then $g(\bm{\bar z})=\varepsilon$. Let $d\in[k]$ be the index satisfying $\alpha_d=\alpha_{\min}$. Then $B(\bm{\bar z},d)=(\bm{y^{\bar z,d}},\bm{\bar z})$ defined in Lemma~\ref{lem:inner} satisfies~\eqref{f_alpha-lb} at equality. Therefore, we deduce that $f_{\valpha}=\alpha_{\min} g+\sum_{j\in[k]}(\alpha_j-\alpha_{\min})f_j$.

From Proposition~\ref{prop:intersection2} applied to~\eqref{P-submod}, we obtain that $\conv(\cP(\vW,\bm{0},\varepsilon))$ is equal to
\[
\left\{(\vy,\vz)\in\R^k\times[0,1]^n:~(y_j,\vz)\in\conv(Q_{f_j}),~\forall j\in[k],~~ \left(y_1+\cdots+y_k,\vz\right)\in\conv(Q_{g})\right\}.
\]
After complementing the $\vz$ variables, we obtain the desired description of $\conv(\cM(\vW,\bm{0},\varepsilon))$. 
This finishes the proof.
\end{proof}

Proposition~\ref{prop:mixing-polymatroid} indicates that if $\bar I(\varepsilon)$ is $\varepsilon$-negligible and $\varepsilon\leq L_{\vW}(\varepsilon)$, then the convex hull of $\mathcal{M}(W,\bm{0},\varepsilon)$ is described by the polymatroid inequalities of $f_j$ with left-hand side $y_j$ for $j\in[k]$ and the polymatroid inequalities of $g$ with left-hand side $\sum_{j\in[k]}y_j$. We have seen in Section~\ref{sec:mixing} that the polymatroid inequalities of $f_j$ with left-hand side $y_j$ for $j\in[k]$ are nothing but the mixing inequalities. In fact, it turns out that an extremal polymatroid inequality of $g$ with left-hand side $\sum_{j\in[k]}y_j$ is either the linking constraint  $y_1+\cdots+y_k\geq \varepsilon$ or an aggregated mixing inequality, depending on whether or not $\bar I(\varepsilon)=[n]$. We consider the $\bar I(\varepsilon)=[n]$ case first.

\begin{proposition}\label{prop:coeff2-special}
Assume that $\bar I(\varepsilon)=[n]$ and $\bar I(\varepsilon)$ is $\varepsilon$-negligible. Then for every extreme point $\vpi$ of $EP_{\tilde{g}}$, the corresponding polymatroid inequality $\sum_{j\in[k]}y_j +\sum_{i\in[n]}\pi_i z_i \geq \varepsilon +\sum_{i\in[n]}\pi_i$ is equivalent to the linking constraint.
\end{proposition}
\begin{proof}
By Theorem~\ref{thm:edmonds}, there exists a permutation $\sigma$ of~$[n]$ such that $\pi_{\sigma(t)}=g(V_{t})-g(V_{t-1})$ where $V_{t}=\{\sigma(1),\ldots,\sigma(t)\}$ for $t\in[n]$ and $V_{0}=\emptyset$. Since $\bar I(\varepsilon)=[n]$ and $\bar I(\varepsilon)$ is $\varepsilon$-negligible, it follows that $g(U)=\varepsilon$ for every $U\subseteq[n]$, so $\pi_{\sigma(t)}=0$ for all $t$. Therefore, $\sum_{j\in[k]}y_j +\sum_{i\in[n]}\pi_i z_i \geq \varepsilon +\sum_{i\in[n]}\pi_i$ equals $\sum_{j\in[k]}y_j \geq \varepsilon$, as required.
\end{proof}

The $\bar I(\varepsilon)\neq[n]$ case is more interesting; the following proposition is similar to Proposition~\ref{prop:coeff1}:

\begin{proposition}\label{prop:coeff2}
Assume that $\bar I(\varepsilon)\neq [n]$ is $\varepsilon$-negligible and $\varepsilon\leq L_{\vW}(\varepsilon)$. Then for every extreme point $\vpi$ of $EP_{\tilde{g}}$, there exists a sequence $\Theta=\left\{i_1\to\cdots\to i_\theta\right\}$ contained in $[n]\setminus \bar I(\varepsilon)$ that satisfies the following:
\begin{enumerate}[{\bf (1)}]
	\item the $j$-mixing-subsequence $\{j_1\to\cdots\to j_{\tau_j}\}$ of $\Theta$ satisfies $w_{j_1,j}=\max\left\{w_{i,j}:i\in[n]\right\}$ for each $j\in[k]$,
	\item the corresponding polymatroid inequality $\sum_{j\in[k]}y_j +\sum_{i\in[n]}\pi_i z_i \geq \varepsilon +\sum_{i\in[n]}\pi_i$ is equivalent to the aggregated mixing inequality~\eqref{aggregate} derived from $\Theta$.
\end{enumerate}
In particular, the polymatroid inequality is of the form
\begin{equation}\label{extremal-aggregate}
	\sum_{j\in[k]}\left(y_j+\sum_{s\in[\tau_j]} (w_{j_s,j}-w_{j_{s+1},j})z_{j_s}\right) - \varepsilon z_{i_\theta}\geq \sum_{j\in[k]} \max\left\{w_{i,j}:i\in[n]\right\}\tag{$\text{A-Mix}^*$}.
\end{equation}
\end{proposition}
\begin{proof}
By Theorem~\ref{thm:edmonds}, there exists a permutation $\sigma$ of~$[n]$ such that $\pi_{\sigma(t)}=g(V_{t})-g(V_{t-1})$ where $V_{t}=\{\sigma(1),\ldots,\sigma(t)\}$ for $t\in[n]$ and $V_{0}=\emptyset$. By Lemma~\ref{lem:negligible}, $g(V_{t})-g(V_{t-1})=g(V_{t}\setminus \bar I(\varepsilon))-g(V_{t-1}\setminus \bar I(\varepsilon))$, so $\pi_{\sigma(t)}$ is nonzero only if $\sigma(t)\not\in \bar I(\varepsilon)$. This in turn implies that at most $|n\setminus \bar I(\varepsilon)|$ coordinates of $\vpi$ are nonzero. Let $\left\{t_1,\ldots,t_\theta\right\}$ be the collection of $t$'s such that $\pi_{\sigma(t)}\neq 0$. Then $1\leq \theta\leq |n\setminus \bar I(\varepsilon)|$. Without loss of generality, we may assume that $t_1>\cdots>t_\theta$. Let $i_1=\sigma(t_1),i_2=\sigma(t_2),\ldots,i_\theta=\sigma(t_\theta)$, and $\Theta$ denote the sequence $\left\{i_1\to\cdots\to i_\theta\right\}$. We will show that $\Theta$ satisfies conditions  {\bf (1)} and {\bf (2)} of the proposition.

{\bf (1):} For $j\in[k]$, let $\Theta_j=\left\{j_1\to\cdots\to j_{\tau_j}\right\}$ denote the $j$-mixing-subsequence of $\Theta$. By definition of the $j$-mixing-subsequence of~$\Theta$, we have $w_{j_1,j}=\max\{w_{i,j}:i\in\Theta\}$. By our choice of $\left\{t_1,\ldots,t_\theta\right\}$ and assumption that $t_1>\cdots>t_\theta$, it follows that $g(V_{t_1})=g([n])$, which means that $f_j(V_{t_1})=f_j([n])$ for each $j\in[k]$. Therefore, we deduce that $\max\{w_{i,j}:i\in\Theta\}=\max\{w_{i,j}:i\in[n]\}$, as required.

{\bf (2):} By convention, we have $w_{i_{\theta+1},j}=w_{j_{\tau_{j}+1},j}=0$ for $j\in[k]$. In addition, due to our choice of $\left\{t_1,\ldots,t_\theta\right\}$, we have  $g(V_{t_{s}})>g(V_{t_s-1})=\cdots=g(V_{t_{s+1}})$ holds for $s<\theta$. Then, we obtain 
\begin{align*}
\pi_{i_s}=\pi_{\sigma(t_s)}=g(V_{t_s})-g(V_{t_{s+1}})&=\sum_{j\in[k]}f_j(V_{t_s})-\sum_{j\in[k]}f_j(V_{t_{s+1}})\\
&=\sum_{j\in[k]}f_j(\left\{i_\theta,i_{\theta-1},\ldots,i_{s}\right\})-\sum_{j\in[k]}f_j(\left\{i_\theta,i_{\theta-1},\ldots,i_{s+1}\right\}) .
\end{align*}
We observed before that $g(V_{t_{s}})>g(V_{t_s-1})=\cdots=g(V_{t_{s+1}})$, so it follows that $f_j(V_{t_{s}})\geq f_j(V_{t_s-1})=\cdots=f_j(V_{t_{s+1}})$, implying in turn that
\[
f_j(\left\{i_\theta,i_{\theta-1},\ldots,i_{s}\right\})-f_j(\left\{i_\theta,i_{\theta-1},\ldots,i_{s+1}\right\})=\left(w_{i_s,j}   -  \max\left\{w_{i,j}:~\text{$i_s$ precedes $i$   in $\Theta$}\right\}   \right)_+.
\]
This means that for $s<\theta$,
\begin{equation}\label{i_s}
\pi_{i_s}=\sum_{j\in[k]}\left(w_{i_s,j}   -  \max\left\{w_{i,j}:~\text{$i_s$ precedes $i$   in $\Theta$}\right\}   \right)_+ .
\end{equation}
Note that
\[
\pi_{i_\theta}=\pi_{\sigma(t_\theta)}=g(V_{t_\theta})-g(V_0)=\sum_{j\in[k]}f_j(V_{t_\theta})-\varepsilon=\sum_{j\in[k]}f_j(\{i_\theta\})-\varepsilon.
\]
Since $f_j(\left\{i_\theta\right\})=w_{i_\theta, j}$ and $\max\left\{w_{i,j}:~\text{$i_\theta$ precedes $i$  in $\Theta$}\right\}$ was set to $w_{j_{\tau_j+1}j}=0$, it follows that
\begin{equation}\label{i_theta}
\pi_{i_\theta}=\sum_{j\in[k]}\left(w_{i_\theta, j}   -  \max\left\{w_{i,j}:~\text{$i_\theta$ precedes $i$  in $\Theta$}\right\}   \right)_+-\varepsilon.
\end{equation}
Therefore, by~\eqref{i_s} and~\eqref{i_theta}, it follows that the polymatroid inequality $\sum_{j\in[k]}y_j +\sum_{i\in[n]}\pi_i z_i \geq \varepsilon +\sum_{i\in[n]}\pi_i$ is precisely~\eqref{extremal-aggregate}. Since $\varepsilon\leq L_{\vW}(\varepsilon)$ by our assumption and $L_{\vW}(\varepsilon)\leq L_{\vW,\Theta}$ by Lemma~\ref{lem:L_Wbound}, $\min\{\varepsilon,~L_{\vW,\Theta}\}=\varepsilon$, and thus the inequality~\eqref{extremal-aggregate} is identical to the aggregated mixing inequality~\eqref{aggregate} derived from $\Theta$, as required.
\end{proof}

\subsection{Necessary conditions for obtaining the convex hull by the mixing and the aggregated mixing inequalities}\label{sec:char}

Let us get back to our original question as to when the convex hull of a joint mixing set with a linking constraint can be completely described by the mixing inequalities and the aggregated mixing inequalities.

By Propositions~\ref{prop:mixing-polymatroid},~\ref{prop:coeff2-special}, and~\ref{prop:coeff2}, if $\bar I(\varepsilon)$ is $\varepsilon$-negligible and $\varepsilon\leq L_{\vW}(\varepsilon)$, then the convex hull of $\cM(\vW,\bm{0},\varepsilon)$ can be described by the mixing and the aggregated mixing inequalities together with the linking constraint $y_1+\cdots+y_k\geq \varepsilon$ and the bounds $\bm{0}\leq \vz\leq\bm{1}$. Another implication of these is that if $\bar I(\varepsilon)$ is $\varepsilon$-negligible and $\varepsilon\leq L_{\vW}(\varepsilon)$, then the aggregated mixing inequalities other than the ones of the form~\eqref{extremal-aggregate} are not necessary.

It turns out that $\bar I(\varepsilon)$ being $\varepsilon$-negligible and $\varepsilon\leq L_{\vW}(\varepsilon)$ are necessary conditions for the mixing and the aggregated mixing inequalities to describe completely the convex hull of $\cM(\vW,\bm{0},\varepsilon)$. Before establishing this result, let us consider examples where either one of these two condition is violated: either $\bar I(\varepsilon)$ is not $\varepsilon$-negligible or $\varepsilon>L_{\vW}(\varepsilon)$.

\begin{example}\label{example3}
	Let us consider Example~\ref{example1} with a slight modification. The following set is the same as~\eqref{eq:example1} except that $w_{4,2}$ is now 3.
	\begin{equation}\label{eq:example3}
	\left\{(\vy,\vz)\in\R_+^2\times\{0,1\}^5:~
	\begin{array}{l}
	y_1+ 8z_1\geq 8,\\
	y_1 + 6z_2\geq 6,\\
	y_1 + 13z_3\geq 13,\\
	y_1 + z_4\geq 1,\\
	y_1 + 4z_5\geq 4,
	\end{array}
	\begin{array}{l}
	y_2 + 3z_1 \geq 3,\\
	y_2 + 4z_2 \geq 4,\\
	y_2 + 2z_3 \geq 2,\\
	\bm{y_2 + 3z_4 \geq 3},\\
	y_2 + z_5 \geq 1,
	\end{array}
	~~y_1+y_2\geq 7\right\}.
	\end{equation}
	In this example, $\bar I(\varepsilon)$ is still $\{4,5\}$. But, $\bar I(\varepsilon)$ is no longer $\varepsilon$-negligible because $3\notin \bar I(\varepsilon)$ yet  $w_{4,2}> w_{3,2}$ implying that condition~\eqref{condition1} is violated. The following set is the same as~\eqref{eq:example1} except that $w_{5,1}$ is now 6.
	\begin{equation}\label{eq:example4}
	\left\{(\vy,\vz)\in\R_+^2\times\{0,1\}^5:~
	\begin{array}{l}
	y_1+ 8z_1\geq 8,\\
	y_1 + 6z_2\geq 6,\\
	y_1 + 13z_3\geq 13,\\
	y_1 + z_4\geq 1,\\
	\bm{y_1 + 6z_5\geq 6},
	\end{array}
	\begin{array}{l}
	y_2 + 3z_1 \geq 3,\\
	y_2 + 4z_2 \geq 4,\\
	y_2 + 2z_3 \geq 2,\\
	y_2 + 2z_4 \geq 2,\\
	y_2 + z_5 \geq 1,
	\end{array}
	~~y_1+y_2\geq 7\right\}.
	\end{equation}
	Again, $\bar I(\varepsilon)$ is  $\{4,5\}$. However, $\bar I(\varepsilon)$ is not $\varepsilon$-negligible because $\sum_{j\in[k]}\max\limits_{i\in\bar I(\varepsilon)}\left\{w_{i,j}\right\}=6+2>\varepsilon$ implying that condition~\eqref{condition2} is violated. Using PORTA~\cite{porta}, one can check that there are facet-defining inequalities other than the mixing and the aggregated mixing inequalities in both of these examples. For instance, $2y_1+3y_2+3z_2+18z_3+3z_4 \geq 38$ is facet-defining for the convex hull of~\eqref{eq:example3} and $2y_1+y_2+z_1+z_2+14z_3+z_4+6z_5 \geq 30$ is facet-defining for the convex hull of~\eqref{eq:example4}.
\end{example}

\begin{example}
In Example~\ref{example2}, $\bar I(\varepsilon)$ is $\varepsilon$-negligible but $\varepsilon>L_{\vW}(\varepsilon)$ (see Examples~\ref{ex1,2-negligible} and~\ref{ex1,2-parameters}). Recall that the convex hull of~\eqref{eq:example2} has a facet-defining inequality, e.g., $7y_1+6y_2+12z_2+49z_3\geq 115$, that is neither a mixing inequality nor an aggregated mixing inequality.
\end{example}

These examples already demonstrate that the mixing and the aggregated mixing inequalities are not sufficient whenever the $\varepsilon$-negligibility condition or the condition $\varepsilon\leq L_{\vW}(\varepsilon)$ does not hold. This is formalized by the following theorem.

\begin{theorem}\label{thm:main}
Let $\vW=\{w_{i,j}\}\in\R_+^{n\times k}$ and $\varepsilon\geq0$. Let $\bar I(\varepsilon)$ and $L_{\vW}(\varepsilon)$ be defined as in~\eqref{indexset} and~\eqref{L_W}, respectively. Then the following statements are equivalent:
\begin{enumerate}[(i)]
	\item  $\bar I(\varepsilon)$ is $\varepsilon$-negligible and $\varepsilon\leq L_{\vW}(\varepsilon)$,
	\item the convex hull of $\cM(\vW,\bm{0},\varepsilon)$ can be described by the mixing inequalities~\eqref{mixing} and the aggregated mixing inequalities~\eqref{aggregate} together with the linking constraint $y_1+\cdots+y_k\geq \varepsilon$ and the bounds $\bm{0}\leq \vz\leq\bm{1}$, and
	\item the convex hull of $\cM(\vW,\bm{0},\varepsilon)$ can be described by the mixing inequalities of the form~\eqref{mixing'} and the aggregated mixing inequalities of the form~\eqref{extremal-aggregate} together with the linking constraint $y_1+\cdots+y_k\geq \varepsilon$ and the bounds $\bm{0}\leq \vz\leq\bm{1}$.
\end{enumerate}
\end{theorem}
The proof of this theorem is given in Appendix~\ref{sec:thm}. Direction {\bf (i)$\Rightarrow$(iii)} is already proved by Propositions~\ref{prop:mixing-polymatroid},~\ref{prop:coeff2-special} and~\ref{prop:coeff2}, and {\bf (iii)$\Rightarrow$(ii)} is trivial. Hence, the main effort in this proof is to establish that {\bf (ii)$\Rightarrow$(i)} holds.

\section{Two-sided chance-constrained programs}\label{sec:two-sided}

A two-sided chance-constrained program has the following form:
\begin{subequations} \label{eq:two-sided-ccp}
	\begin{align}
	\min_{\vx\in \cX}\quad& \vh^\top \vx\label{two-sided-ccp-obj}\\
	\st\quad& \P\left[\ \left|\va^\top \vx-b(\vxi)\right|\leq\vc^\top\vx -d(\vxi)\right]\geq 1-\epsilon,\label{two-sided-ccp-cons}
	\end{align}
\end{subequations}
where $\cX\subseteq \R^m$ is a domain for the decision variables $\vx$, $\epsilon\in(0,1)$ is a risk level, $\va,\vc\in\R^m$ are deterministic coefficient vectors, and $b(\vxi),d(\vxi)\in\R$ are random parameters that depend on the random variable $\vxi\in \bm{\Xi}$; see  \citet{liu2018intersection} for further details on two-sided CCPs. Note that \eqref{eq:two-sided-ccp} is indeed a special case of joint CCPs with random right-hand vector because the nonlinear constraint in~\eqref{two-sided-ccp-cons} is equivalent to the following system of two linear inequalities:
\begin{align}
(\vc +\va)^\top \vx &\geq d(\vxi)+b(\vxi),\tag{28b'}\label{two-sided-ccp-cons'}\\
(\vc -\va)^\top \vx &\geq d(\vxi)-b(\vxi).\tag{28b''}\label{two-sided-ccp-cons''}
\end{align}
Hence, just like other joint CCPs we have studied in this paper, the two-sided CCP can be reformulated as a mixed-integer linear program. Note also that in the resulting MILP formulation, each inequality~\eqref{two-sided-ccp-cons'} and~\eqref{two-sided-ccp-cons''} individually will  lead to a mixing set, and consequently the resulting MILP reformulation will have a substructure containing the intersection of two mixing sets where the continuous variables of these mixing sets are correlated. %
Recall that the convex hull of the intersection of two mixing sets, as long as they do not share continuous variables, can be completely described by the mixing inequalities. However, as we observed in Section~\ref{sec:aggregated} the mixing inequalities are not sufficient when additional constraints linking the continuous variables are present. We have thus far considered constraints on the continuous variables that correspond to quantile cuts. On the other hand, \citet{liu2018intersection} focus on additional bound constraints on the continuous variables that can be easily justified when for example the original decision variables $\vx$ are bounded. In particular, they use bounds on $\vc^\top\vx$ and $\va^\top\vx$. To simplify our discussion
\footnote{Arbitrary bounds on~$\vc^\top\vx$ and $\va^\top\vx$ can be also dealt with by taking appropriate linear transformations (see Section 1.1 of~\cite{liu2018intersection}).}, 
let us assume that
\begin{equation}\label{two-sided:bounds}
\vc^\top\vx\geq 0,\quad u_{\va}\geq \va^\top\vx\geq 0,\quad \forall \vx\in\cX.
\end{equation}
In order to point out the intersection of two mixing sets connection and also to explain how to use~\eqref{two-sided:bounds} to strengthen this intersection, we follow the setup in~\cite{liu2018intersection} and define two continuous variables $y_{\vc}$ and $y_{\va}$ for $\vc^\top \vx$ and $\va^\top \vx$, respectively\footnote{This is equivalent to taking continuous variables $y_{\vc+\va}=(\vc+\va)^\top\vx$ and $y_{\vc-\va}=(\vc-\va)^\top\vx$ as in~\eqref{eq:ccp-re}.}. Given $n$ scenarios $\vxi^1,\ldots,\vxi^n$, define $w_i:=d(\vxi^i)+b(\vxi^i)$ and $v_i:=d(\vxi^i)-b(\vxi^i)$ for $i\in[n]$. \citet{liu2018intersection} focus on the setting where the following condition holds:
\begin{equation}\label{two-sided:condition}
u_{\va}\geq \max\left\{w_i:~i\in[n]\right\},\quad w_i\geq v_i\geq 0,\quad\forall i\in[n]
\end{equation}
In particular, as the parameters $w_i,v_i$ are nonnegative for all $i\in[k]$, the MIP reformulation, given by~\eqref{eq:ccp-re}, of~\eqref{eq:two-sided-ccp} gives rise to the following mixed-integer set:
\begin{subequations}\label{eq:two-sided}
\begin{align}
&y_{\vc}+y_{\va} +w_i z_i \geq w_i,~~\qquad \qquad\qquad\forall i\in[n],\label{two-bigM-w}\\
&y_{\vc}-y_{\va} +(v_i+u_{\va})z_i \geq v_i,\qquad \qquad\forall i\in[n],\label{two-bigM-v}\\
&u_{\va}\geq y_{\va}\geq 0,\label{two-lb}\\
&y_{\vc}\geq 0,\label{two-linking}\\
&\vz\in\{0,1\}^n.\label{two-binary}
\end{align}
\end{subequations}
Note that the coefficient of $z_i$ in~\eqref{two-bigM-v} differs from the right-hand side, but~\eqref{two-bigM-v} indeed corresponds to the mixing set for~\eqref{two-sided-ccp-cons''} since $u_{\va}$ can be added to the both sides of~\eqref{two-bigM-v} and $y_{\vc}-y_{\va}+u_{\va}\geq 0$ holds. In addition,~\eqref{two-bigM-w} corresponds to the mixing set for~\eqref{two-sided-ccp-cons'} since $y_{\vc}+y_{\va}\geq 0$ by~\eqref{two-lb} and~\eqref{two-linking}. \citet{liu2018intersection} characterize the convex hull description of the mixed-integer set given by~\eqref{eq:two-sided} under the condition~\eqref{two-sided:condition}. It turns out that this result  can be driven as a simple consequence of Theorem~\ref{thm:main}. We will  elaborate on this in the remainder of this section.

After setting $y_1=y_{\vc}+y_{\va}$ and $y_2=y_{\vc}-y_{\va}+u_{\va}$, the set \eqref{eq:two-sided} is equivalent to the following system:
\begin{subequations}
\begin{align}
&y_1 +w_i z_i \geq w_i,~~\quad\qquad\qquad\qquad \qquad\forall i\in[n],\label{two-bigM-w'}\\
&y_2 +(v_i+u_{\va})z_i \geq (v_i+u_{\va}),\qquad \qquad\forall i\in[n],\label{two-bigM-v'}\\
&u_{\va}\geq y_1-y_2\geq -u_{\va},\label{two-lb'}\\
&y_1+y_2\geq u_{\va},\ y_1\geq0, y_2\geq0\label{two-linking'}\\
&\vz\in\{0,1\}^n.\label{two-binary'}
\end{align}
\end{subequations}
Note that the set defined by~\eqref{two-bigM-w'},~\eqref{two-bigM-v'},~\eqref{two-linking'}, and~\eqref{two-binary'} is nothing but a joint mixing set with  a linking constraint of the form $\cM(\vW,\bm{0},u_{\va})$ with $k=2$. Moreover, it follows from $w_i,v_i\geq0$ for $i\in[N]$ that
\[
w_i+(v_i+u_{\va})\geq u_{\va}~~\text{for all}~~i\in[n]\quad\text{and}\quad \min\left\{w_i:~i\in[n]\right\}+\min\left\{v_i+u_{\va}:~i\in[n]\right\}\geq u_{\va}.
\]
Then $\bar I(u_{\va})$, where $\bar I$ is defined as in~\eqref{indexset}, is given by
$$\bar I(u_{\va}) = \left\{i\in[n]:\ w_i=v_i=0\right\}=\left\{i\in[n]:\ w_i=0,\ v_i+u_{\va}=u_{\va}\right\}.$$
If $\bar I(u_{\va})\neq \emptyset$, then 
$$\max\limits_{i\in \bar I(u_{\va})}\left\{w_i\right\}=0\quad\text{and}\quad \max\limits_{i\in \bar I(u_{\va})}\left\{v_i+u_{\va}\right\}=u_{\va},$$
in which case conditions~\eqref{condition1} and~\eqref{condition2} are clearly satisfied. Therefore, by Definition~\ref{def:eps-negligible}, $\bar I(u_{\va})$ is $u_{\va}$-negligible. Furthermore, we can next argue that $L_{\vW}(u_{\va})\geq u_{\va}$. By the definition of $L_{\vW}(u_{\va})$ given in~\eqref{L_W}, when $\bar I(u_{\va})\neq [n]$,
$$L_{\vW}(u_{\va})\geq \min_{i\in[n]}\left\{w_i\right\}+\min_{i\in[n]}\left\{v_i+u_{\va}\right\}\geq u_{\va},$$
and we have $L_{\vW}(u_{\va})=+\infty \geq u_{\va}$ if $\bar I(u_{\va})=[n]$. Hence, by Theorem~\ref{thm:main}, the convex hull of the joint mixing set with a linking constraint can be obtained after applying the mixing and the aggregated mixing inequalities.

In particular, given a sequence $\{i_1\dots\to i_\theta\}$ of indices in $[n]$, the corresponding aggregated mixing inequality~\eqref{aggregate} is of the following form:
\begin{equation}\label{generalized'}
y_1+y_2 + \sum_{s\in[\tau_R]}(w_{r_s}-w_{r_{s+1}}) z_{r_s} +\sum_{s\in[\tau_G]}(v_{g_s}-v_{g_{s+1}}) z_{g_s} -u_{\va}z_{i_\theta}\geq w_{r_1}+(v_{g_1}+u_{\va}),
\end{equation}
where $\{r_1\to\cdots \to r_{\tau_R}\}$ and $\{g_1\to\cdots \to g_{\tau_G}\}$ are the $1$-mixing-subsequence and the $2$-mixing-subsequence of $\Theta$, respectively, and $w_{r_{\tau_R+1}}:=0$, $v_{g_{\tau_G+1}}:=-u_{\va}$. By Lemma~\ref{lem:j-sequence}, we know that $z_{g_{\tau_G}}=z_{i_\theta}$, so $(v_{g_{\tau_G}}-v_{g_{\tau_G+1}}) z_{g_{\tau_G}} -u_{\va}z_{i_\theta}=v_{g_{\tau_G}}z_{g_{\tau_G}}$. Since $y_1+y_2=2y_{\vc}+u_{\va}$,~\eqref{generalized'} is equivalent to the following inequality:
\begin{equation}\label{generalized}
2y_p + \sum_{s\in[\tau_R]}(w_{r_s}-w_{r_{s+1}}) z_{r_s} +\sum_{s\in[\tau_G]}(v_{g_s}-v_{g_{s+1}}) z_{g_s}\geq w_{r_1}+v_{g_1},
\end{equation}
where $w_{r_{\tau_R+1}}:=0$ as before but $v_{g_{\tau_G+1}}$ is now set to $0$.

In~\cite{liu2018intersection}, the inequality~\eqref{generalized} is called the {\it generalized mixing inequality from $\Theta$}, so the aggregated mixing inequalities generalize the generalized mixing inequalities to arbitrary $k$. Furthermore, Theorem~\ref{thm:main} can be extended slightly to recover the following main result of~\cite{liu2018intersection}:

\begin{theorem}[\cite{liu2018intersection}, Theorem 3.1]
Let $\cP$ be the mixed-integer set defined by~\eqref{two-bigM-w'}--\eqref{two-binary'}. Then  the convex hull of $\cP$ can be described by the mixing inequalities for $y_1,y_2$, the aggregated mixing inequalities of the form~\eqref{generalized'} together with~\eqref{two-lb'} and the bounds $\bm{0}\leq \vz\leq\bm{1}$ under the condition~\eqref{two-sided:condition}.%
\end{theorem}
\begin{proof}
Let $\cR$ be the mixed-integer set defined by~\eqref{two-bigM-w'},~\eqref{two-bigM-v'},~\eqref{two-linking'}, and~\eqref{two-binary'}. Then $\cP\subseteq \cR$ and, by Theorem~\ref{thm:main}, $\conv(\cR)$ is described by the mixing inequalities for $y_1,y_2$ and the generalized mixing inequalities of the form~\eqref{generalized'} together with $\bm{0}\leq \vz\leq\bm{1}$. We will argue that adding constraint~\eqref{two-lb'}, that is $u_{\va}\geq y_1-y_2\geq -u_{\va}$, to the description of~$\conv(\cR)$ does not affect integrality of the resulting system.

By Lemma~\ref{lem:inner}, the extreme rays of $\conv(\cR)$ are $(\bm{e^j},\bm{0})$ for $j\in[2]$, and the extreme points are
\begin{itemize}
	\item $A(z)=(y_1,y_2,\vz)$ for $\vz\in \{0,1\}^n\setminus\{\bm{1}\}$ where 
	\[
	y_1=\max\limits_{i\in[n]}\left\{w_i(1-z_i)\right\}\quad\text{and}\quad y_2=\max\limits_{i\in[n]}\left\{(v_i+u_{\va})(1-z_i)\right\},\]
	\item $B(1)=(u_{\va},0,\bm{1})$ and $B(2)=(0,u_{\va},\bm{1})$.
\end{itemize}
It follows from \eqref{two-sided:condition} %
that all extreme points of $\conv(\cR)$ satisfy $u_{\va}\geq y_1-y_2\geq -u_{\va}$. Observe that two hyperplanes $\left\{(y,z):u_{\va}=y_1-y_2\right\}$ and $\left\{(y,z):y_1-y_2=-u_{\va}\right\}$ are parallel. So, each of the new extreme points created after adding $u_{\va}\geq y_1-y_2\geq -u_{\va}$ is obtained as the intersection of one of the two hyperplanes and a ray emanating from an extreme point of $\conv(\cR)$. Since every extreme ray of $\conv(\cR)$ has $\bm{0}$ in its $z$ component and every extreme point of $\conv(\cR)$ has integral $z$ component, the $z$ component of every new extreme point is also integral, as required.

Therefore, the convex hull of $\cP$ is equal to $\left\{(\vy,\vz)\in\conv(\cR):(\vy,\vz)~\text{satisfies}~\eqref{two-lb'}\right\}$, implying in turn that $\conv(\cP)$ can be described by the mixing inequalities for $y_1,y_2$, the aggregated mixing inequalities of the form~\eqref{generalized'} together with~\eqref{two-lb'} and the bounds $\bm{0}\leq \vz\leq\bm{1}$, as required.
\end{proof}

\section{Conclusions}
In this paper, we show that mixing inequalities with binary variables may be viewed as polymatroid inequalities applied to a specific submodular function. With this observation, we unify and generalize extant valid inequalities and convex hull descriptions of the mixing sets with common binary variables and their intersection under additional constraints on a linear function of the continuous variables. Such substructures have attracted interest as they appear in joint CCPs. 
However, our results are broadly applicable to other settings that involve similar substructures, including epigraphs of submodular functions other than those considered in this paper.

\section*{Acknowledgments}
The authors wish to thank the review team for their constructive feedback that improved the presentation of the material in this paper. This research is supported, in part, by ONR grant N00014-19-1-2321 and the Institute for Basic Science (IBS-R029-C1, IBS-R029-Y2).

{
\bibliographystyle{abbrvnat} %
\bibliography{mybibfile}

}

\appendix
\section{Proof of Lemma~\ref{lem:aggregate-lemma}}

Let $(\bm{\bar y},\bm{\bar z})\in\R_+^k\times [0,1]^n$ be a point satisfying~\eqref{bigM'}--\eqref{linking'}, and assume that $(\bm{\bar y},\bm{\bar z})$ satisfies~\eqref{aggregate} for all sequences contained in $\left\{i\in[n]:~\bar z_i<1\right\}$. Then we need to prove that $(\bm{\bar y},\bm{\bar z})$ satisfies~\eqref{aggregate} for all the other sequences as well.

For a sequence $\Theta$, we denote by $N(\Theta)$ the set $\left\{i\in\Theta:~\bar z_i=1\right\}$. We argue by induction on $\left|N(\Theta)\right|$ that $(\bm{\bar y},\bm{\bar z})$ satisfies~\eqref{aggregate} for $\Theta$. If $\left|N(\Theta)\right|=0$, then $(\bm{\bar y},\bm{\bar z})$ satisfies~\eqref{aggregate} by the assumption. For the induction step, we assume that $(\bm{\bar y},\bm{\bar z})$ satisfies~\eqref{aggregate} for every sequence $\Theta$ with $|N(\Theta)|<N$ for some $N\geq 1$. Now we take a sequence $\Theta=\left\{i_1\to\cdots\to i_\theta\right\}$ with $|N(\Theta)|=N$. Notice that $(\bm{\bar y},\bm{\bar z})$ satisfies~\eqref{aggregate} if and only if $(\bm{\bar y},\bm{\bar z})$ satisfies
\begin{align}
&\sum_{j\in[k]}\left(\bar y_j+\sum_{t\in[\theta]} \left(w_{i_t,j}   -  \max\left\{w_{i,j}:~
\text{$i_t$ precedes $i$ in $\Theta$}\right\}   \right)_+\bar z_{i_t}\right)\nonumber\\
&\qquad\qquad \qquad\qquad\qquad\qquad \qquad\qquad\qquad-\min\left\{\varepsilon,~L_{\vW,\Theta}\right\} \bar z_{i_\theta}\geq\sum_{j\in[k]} \max\left\{w_{i,j}:~i\in\Theta\right\}.\label{eq1'}
\end{align}
Hence, it is sufficient to show that $(\bar  y,\bm{\bar z})$ satisfies~\eqref{eq1'}. We consider two cases $\bar z_{i_\theta}=1$ and $\bar z_{i_\theta}\neq 1$ separately.

First, consider the case when $\bar z_{i_\theta}\neq 1$. Since $|N(\Theta)|\geq 1$, we have $\bar z_{i_p}=1$ for some $p\in[\theta-1]$. Let $\Theta^\prime$ denote the subsequence of $\Theta$ obtained by removing $i_p$. Then $|N(\Theta^\prime)|=|N(\Theta)|-1$, so it follows from the induction hypothesis that~\eqref{aggregate} for $\Theta^\prime$ is valid for $(\bm{\bar y},\bm{\bar z})$:
\begin{align}
&\sum_{j\in[k]}\left(\bar y_j+\sum_{t\in[\theta]\setminus\{p\}} \left(w_{i_t,j}   -  \max\left\{w_{i,j}:~
\text{$i_t$ precedes $i$ in $\Theta^\prime$}\right\}   \right)_+\bar z_{i_t}\right)\nonumber\\
&\qquad\qquad \qquad\qquad\qquad\qquad \qquad\qquad\qquad-\min\left\{\varepsilon,~L_{\vW,\Theta^\prime}\right\} \bar z_{i_\theta}\geq\sum_{j\in[k]} \max\left\{w_{i,j}:~i\in\Theta^\prime\right\}.\label{eq2}
\end{align}
Since $\Theta^\prime$ is a subsequence of $\Theta$, it follows that for any $t\neq p$.
\begin{equation}\label{eq3}
\left(w_{i_t,j}   -  \max\left\{w_{i,j}:~
\text{$i_t$ precedes $i$ in $\Theta^\prime$}\right\}   \right)_+\geq \left(w_{i_t,j}   -  \max\left\{w_{i,j}:~
\text{$i_t$ precedes $i$ in $\Theta$}\right\}   \right)_+.
\end{equation}
Since $-\bar z_{i_t}\geq -1$ is valid for each $t$, we deduce the following inequality from~\eqref{eq2}:
\begin{align}
&\sum_{j\in[k]}\left(\bar y_j+\sum_{t\in[\theta]\setminus\{p\}} \left(w_{i_t,j}   -  \max\left\{w_{i,j}:~
\text{$i_t$ precedes $i$ in $\Theta$}\right\}   \right)_+\bar z_{i_t}\right)-\min\left\{\varepsilon,~L_{\vW,\Theta^\prime}\right\} \bar z_{i_\theta}\nonumber\\
&~~~\quad\qquad \qquad\qquad\qquad\geq\sum_{j\in[k]} \max\left\{w_{i,j}:~i\in\Theta\right\}-\sum_{j\in[k]}\left(w_{i_p,j}   -  \max\left\{w_{i,j}:~
\text{$i_p$ precedes $i$ in $\Theta$}\right\}   \right)_+,\label{eq4}
\end{align}
because 
\[
\sum_{t\in[\theta]\setminus\{p\}} \left(w_{i_t,j}   -  \max\left\{w_{i,j}:~
\text{$i_t$ precedes $i$ in $\Theta^\prime$}\right\}   \right)_+=\sum_{j\in[k]} \max\left\{w_{i,j}:~i\in\Theta^\prime\right\}
\]
and
\begin{equation}\label{sum}
\sum_{t\in[\theta]} \left(w_{i_t,j}   -  \max\left\{w_{i,j}:~
\text{$i_t$ precedes $i$ in $\Theta$}\right\}   \right)_+=\sum_{j\in[k]} \max\left\{w_{i,j}:~i\in\Theta\right\}.
\end{equation}
Moreover, notice that $L_{\vW,\Theta^\prime}\geq L_{\vW,\Theta}$ due to~\eqref{eq3}. So, it follows that~\eqref{eq4} implies~\eqref{eq1'} since $\bar z_{i_p}=1$. This in turn implies that $(\bm{\bar y},\bm{\bar z})$ satisfies~\eqref{aggregate} for $\Theta$, as required.

Next we consider the $\bar z_{i_\theta}=1$ case. In this case,~\eqref{eq1'} is equivalent to
\begin{align}
&\sum_{j\in[k]}\left(\bar y_j+\sum_{t\in[\theta-1]} \left(w_{i_t,j}   -  \max\left\{w_{i,j}:~
\text{$i_t$ precedes $i$ in $\Theta$}\right\}   \right)_+\bar z_{i_t}\right)\nonumber\\
&\qquad\qquad \qquad\qquad\qquad\qquad \qquad\qquad\geq \min\left\{\varepsilon,~L_{\vW,\Theta}\right\}-\sum_{j\in[k]}w_{i_\theta, j}+\sum_{j\in[k]} \max\left\{w_{i,j}:~i\in\Theta\right\}.\label{eq5}
\end{align}
Take the subsequence $\Theta^\prime$ of $\Theta$ obtained by removing $i_\theta$. As before, we have $|N(\Theta^\prime)|=|N(\Theta)|-1$, and the induction hypothesis implies that~\eqref{aggregate} for $\Theta^\prime$ is valid for $(\bm{\bar y},\bm{\bar z})$:
\begin{align}
&\sum_{j\in[k]}\left(\bar y_j+\sum_{t\in[\theta-2]} \left(w_{i_t,j}   -  \max\left\{w_{i,j}:~
\text{$i_t$ precedes $i$ in $\Theta^\prime$}\right\}   \right)_+\bar z_{i_t}\right)\nonumber\\
&\qquad\qquad \qquad\qquad\qquad\qquad +\left(\sum_{j\in[k]}w_{i_{\theta-1},j}-\min\left\{\varepsilon,~L_{\vW,\Theta^\prime}\right\} \right)\bar z_{i_{\theta-1}}\geq\sum_{j\in[k]} \max\left\{w_{i,j}:~i\in\Theta^\prime\right\}.\label{eq6}
\end{align}
We will deduce from~\eqref{eq6} that~\eqref{eq5} is valid for $(\bm{\bar y},\bm{\bar z})$. As $\Theta^\prime$ is a subsequence of $\Theta$,~\eqref{eq3} holds for $t\in[\theta-2]$. So, as $(\bm{\bar y},\bm{\bar z})$ satisfies $-\bar z_{i_t}\geq -1$ for $t\in[\theta-2]$, we obtain the following from~\eqref{eq6}:
\begin{align}
&\sum_{j\in[k]}\left(\bar y_j+\sum_{t\in[\theta-2]} \left(w_{i_t,j}   -  \max\left\{w_{i,j}:
\text{$i_t$ precedes $i$ in $\Theta$}\right\}   \right)_+\bar z_{i_t}\right)+ \left(\sum_{j\in[k]}w_{i_{\theta-1},j}-\min\left\{\varepsilon,L_{\vW,\Theta^\prime}\right\} \right)\bar z_{i_{\theta-1}}\nonumber\\
&\qquad\qquad\qquad\qquad\qquad\qquad\qquad\geq\sum_{j\in[k]}\min\left\{w_{i_{\theta-1},j},w_{i_\theta, j}\right\} -\sum_{j\in[k]} w_{i_{\theta},j}+\sum_{j\in[k]} \max\left\{w_{i,j}:~i\in\Theta\right\},\label{eq7}
\end{align}
because~\eqref{sum} holds,
\[
\sum_{t\in[\theta-1]} \left(w_{i_t,j}   -  \max\left\{w_{i,j}:~
\text{$i_t$ precedes $i$ in $\Theta^\prime$}\right\}   \right)_+=\sum_{j\in[k]} \max\left\{w_{i,j}:~i\in\Theta^\prime\right\},
\]
and
\begin{equation}\label{sum'}
\sum_{j\in[k]}w_{i_{\theta-1},j}-\sum_{j\in[k]}\left(w_{i_{\theta-1},j}   -  \max\left\{w_{i,j}:~
\text{$i_{\theta-1}$ precedes $i$ in $\Theta$}\right\}   \right)_+=\sum_{j\in[k]}\min\left\{w_{i_{\theta-1},j},w_{i_\theta, j}\right\}.
\end{equation}
Now let us compare the coefficient of $\bar z_{i_{\theta-1}}$ in~\eqref{eq7} and that of $\bar z_{i_{\theta-1}}$ in~\eqref{eq5}. If the coefficient in~\eqref{eq7} is less than the coefficient in~\eqref{eq5}, then~\eqref{eq7} implies that~\eqref{eq5} is valid, because we can add an appropriate scalar multiple of $\bar z_{i_\theta-1}\geq 0$ to~\eqref{eq7} in order to achieve the coefficient in~\eqref{eq5} and the term $\sum_{j\in[k]}\min\left\{w_{i_{\theta-1},j},w_{i_\theta, j}\right\}$ in the right-hand side of~\eqref{eq7} is at least $L_{\vW,\Theta}$. If not, then by adding an appropriate scalar multiple of $-\bar z_{i_{\theta-1}}\geq -1$ to~\eqref{eq7}, we deduce the following inequality:
\begin{align}
&\sum_{j\in[k]}\left(\bar y_j+\sum_{t\in[\theta-1]} \left(w_{i_t,j}   -  \max\left\{w_{i,j}:~
\text{$i_t$ precedes $i$ in $\Theta$}\right\}   \right)_+\bar z_{i_t}\right)\nonumber\\
&\qquad\qquad \qquad\qquad\qquad\qquad \qquad\qquad\geq \min\left\{\varepsilon,~L_{\vW,\Theta^\prime}\right\}-\sum_{j\in[k]}w_{i_\theta, j}+\sum_{j\in[k]} \max\left\{w_{i,j}:~i\in\Theta\right\},\label{eq8}
\end{align}
because~\eqref{sum'} holds. Since $\Theta^\prime$ is a subsequence of $\Theta$, we have $L_{\vW,\Theta^\prime}\geq L_{\vW,\Theta}$, so it follows that the term $\min\left\{\varepsilon,~L_{\vW,\Theta^\prime}\right\}$ in the right-hand side of~\eqref{eq8} is at least $\min\left\{\varepsilon,~L_{\vW,\Theta}\right\}$. Therefore,~\eqref{eq8} implies that~\eqref{eq5} is valid for $(\bm{\bar y},\bm{\bar z})$. In summary, when $\bar z_{i_\theta}=1$, $(\bm{\bar y},\bm{\bar z})$ satisfies~\eqref{eq5}, thereby proving that $(\bm{\bar y},\bm{\bar z})$ satisfies~\eqref{aggregate}. This finishes the proof of Lemma~\ref{lem:aggregate-lemma}.

\section{Proof of Theorem~\ref{thm:main}}\label{sec:thm}

Propositions~\ref{prop:mixing-polymatroid},~\ref{prop:coeff2-special} and~\ref{prop:coeff2} already prove that {\bf (i)$\Rightarrow$(iii)},  and the direction {\bf (iii)$\Rightarrow$(ii)} is trivial. Thus, what remains is to show {\bf (ii)$\Rightarrow$(i)}. We will prove the contrapositive of this direction. It is sufficient to exhibit a point $(\bm{\bar y}, \bm{\bar z})$ with $\sum_{j\in[k]}\bar y_j\geq \varepsilon$ and $\mathbf{0}\leq \bm{\bar z}\leq\mathbf{1}$ that satisfies the mixing and the aggregated mixing inequalities but is not contained in the convex hull of $\cM(\vW,\mathbf{0},\varepsilon)$.
	
	Assume first that $\bar I(\varepsilon)$ is not $\varepsilon$-negligible. Then $\bar I(\varepsilon)$ is nonempty and either~\eqref{condition1} or~\eqref{condition2} is violated. First, consider the case when~\eqref{condition2} is violated. Take a minimal subset $U$ of $\bar I(\varepsilon)$ satisfying $\sum_{j\in[k]}\max\limits_{i\in U}\{w_{i,j}\}>\varepsilon$. Note that by definition of $\bar I(\varepsilon)$, we have for every $i\in \bar I(\epsilon)$ that $\sum_{j\in[k]}w_{i,j}\leq\varepsilon$. Then by the assumption that $\sum_{j\in[k]}\max\limits_{i\in U}\{w_{i,j}\}>\varepsilon$, we deduce that $|U|\geq 2$. Moreover, by our minimal choice of $U$, we have  $\sum_{j\in[k]}\max\limits_{i\in V}\{w_{i,j}\}\leq\varepsilon$ for any $V\subset U$ such that $|V|\leq|U|-1$. Moreover, for each $j\in[k]$, the largest element of $\left\{w_{i,j}:i\in U\right\}$ is contained in $|U|-1$ subsets of $\left\{w_{i,j}:i\in U\right\}$ of size $|U|-1$, while the second largest element of $U$  is the largest in another subset of size $|U|-1$. From these observations, we deduce that
	\begin{equation}\label{eq2-1}
	(|U|-1)\sum_{j\in[k]}\max\limits_{i\in U}\{w_{i,j}\}+\sum_{j\in[k]}\text{second-max}\{w_{i,j}:~i\in U\}=\sum_{\substack{V\subset U\\|V|=|U|-1}}	\left(\sum_{j\in[k]}\max\limits_{i\in V}\{w_{i,j}\}\right)\leq |U|\,\varepsilon
	\end{equation}
	where $\text{second-max}\{w_{i,j}:~i\in U\}$ denotes the second largest value in $\{w_{i,j}:~i\in U\}$ for $j\in[k]$. Let us consider $(\bm{\bar y},\bm{\bar z})$ where
	\[
	\bar z_i=\begin{cases}
	\frac{1}{|U|}&\text{if}~i\in U\\
	1&\text{if}~i\notin U
	\end{cases}
	\quad
	\text{and}
	\quad
	\bar y_j=\begin{cases}
	\frac{|U|-1}{|U|}\max\limits_{i\in U}\left\{w_{i,j}\right\},&\text{if}~j\in[k-1]\\
	\frac{|U|-1}{|U|}\max\limits_{i\in U}\left\{w_{i,k}\right\}+\left(\varepsilon-\frac{|U|-1}{|U|}\sum\limits_{j\in[k]}\max\limits_{i\in U}\left\{w_{i,j}\right\}\right),&\text{if}~j=k
	\end{cases}
	\]
	Then, we always have $\sum_{j\in[k]}\bar y_j=\varepsilon$. This, together with~\eqref{eq2-1}, implies that
	\begin{equation}\label{eq2-1'}
	\sum_{j\in[k]}\bar y_j=\varepsilon \geq \frac{|U|-1}{|U|}\sum\limits_{j\in[k]}\max\limits_{i\in U}\left\{w_{i,j}\right\}+\frac{1}{|U|}\sum_{j\in[k]}\text{second-max}\{w_{i,j}:~i\in U\}.
	\end{equation}
	Then, from $\vW\in\R^{n\times k}_+$ we deduce $\sum_{j\in[k]}\text{second-max}\{w_{i,j}:~i\in U\}\geq 0$, and hence $\bar y_k\geq \frac{|U|-1}{|U|}\max\limits_{i\in U}\left\{w_{i,k}\right\}$. 
	Let us argue that $(\bm{\bar y},\bm{\bar z})$ satisfies the mixing and the aggregated mixing inequalities. Take a $j$-mixing-sequence $\{j_1\to\cdots\to j_{\tau_j}\}$. Since $\sum_{s\in[\tau_j]}(w_{j_s,j}-w_{j_{s+1},j})=w_{j_1,j}$, $(\bm{\bar y},\bm{\bar z})$ satisfies~\eqref{mixing} if and only if 
	\[
	\bar y_j \geq \frac{|U|-1}{|U|}\sum_{j_s\in U}\left( w_{j_s,j}-w_{j_{s+1},j}\right).
	\]
	As $\sum_{j_s\in U}\left( w_{j_s,j}-w_{j_{s+1},j}\right)\leq \max\limits_{i\in U}\{w_{i,j}\}$, it follows that $(\bm{\bar y},\bm{\bar z})$ satisfies~\eqref{mixing}. Now we argue that $(\bm{\bar y},\bm{\bar z})$ satisfies every aggregated mixing inequality. By Lemma~\ref{lem:aggregate-lemma}, it is sufficient to argue this for only the sequences $\Theta=\left\{i_1\to\cdots\to i_\theta\right\}$ that are contained in $U$. By~\eqref{coeff-transf}, $(\bm{\bar y},\bm{\bar z})$ satisfies~\eqref{aggregate} for $\Theta$ if and only if
	\begin{align}
	&\sum_{j\in[k]}\left(\bar y_j+\sum_{t\in[\Theta]} \left(w_{i_t,j}   -  \max\left\{w_{i,j}:~
	\text{$i_t$ precedes $i$ in $\Theta$}\right\}   \right)_+\bar z_{i_t}\right)\nonumber\\
	&\qquad\qquad \qquad\qquad\qquad\qquad \qquad\qquad\qquad-\min\left\{\varepsilon,~L_{\vW,\Theta}\right\} \bar z_{i_\theta}\geq\sum_{j\in[k]} \max\left\{w_{i,j}:~i\in\Theta\right\}.\label{eq1}
	\end{align}
	Since $\Theta\subseteq U$, we have $\bar z_{i_1}=\cdots=\bar z_{i_\theta}=\frac{1}{|U|}$. Then,~\eqref{eq1} is exactly 
	\begin{equation}
	\sum_{j\in[k]}\bar y_j\geq\frac{|U|-1}{|U|}\sum_{j\in[k]} \max\left\{w_{i,j}:~i\in\Theta\right\}+\frac{1}{|U|}\min\left\{\varepsilon,~L_{\vW,\Theta}\right\}.\label{eq2-2}
	\end{equation}
	Recall that $\sum_{j\in[k]}\bar y_j=\varepsilon$. If $|\Theta|=1$, then because $|U|\geq 2$ we deduce $\Theta\neq U$. Moreover, because $|\Theta|=1$ and $\Theta$ is a proper subset of $\bar I (\epsilon)$, we deduce from the definition of $\bar I (\epsilon)$ that $\sum_{j\in[k]} \max\left\{w_{i,j}:~i\in\Theta\right\}\leq \varepsilon$. Hence, when $|\Theta|=1$, we also have $\min\left\{\varepsilon,~L_{\vW,\Theta}\right\}\leq \varepsilon$, and thus~\eqref{eq2-2} clearly holds. So, we may assume that $|\Theta|\geq 2$. By definition of $L_{\vW,\Theta}$ in~\eqref{L_WTheta}, we have $L_{\vW,\Theta}\leq \sum_{j\in[k]}\text{second-max}\{w_{i,j}:i\in \Theta\}$ where $\text{second-max}\{w_{i,j}:i\in \Theta\}$ denotes the second largest element in $\{w_{i,j}:~i\in \Theta\}$. Since $\max\limits_{i\in \Theta}\left\{w_{i,j}\right\}\leq \max\limits_{i\in U}\left\{w_{i,j}\right\}$ and $\text{second-max}\{w_{i,j}:~i\in \Theta\}\leq \text{second-max}\{w_{i,j}:~i\in U\}$ hold because $\Theta\subseteq U$, we deduce from~\eqref{eq2-1'} that~\eqref{eq2-2} holds. Consequently, Lemma~\ref{lem:aggregate-lemma} implies that $(\bm{\bar y},\bm{\bar z})$ satisfies the aggregated mixing inequalities~\eqref{aggregate} for all sequences as well. Let us now  show that $(\bm{\bar y},\bm{\bar z})$ is not contained in $\conv(\cM(\vW,\mathbf{0},\varepsilon))$. Observe that $(\bm{\bar y},\bm{\bar z})$ satisfies the constraints $z_i\leq 1$ for $i\notin U$ at equality. For $j\in[k-1]$, let $\{j_1\to\cdots\to j_{|U|}\}$ be an ordering of the indices in $U$ such that $w_{j_1,j}\geq\cdots \geq w_{j_{|U|},j}$. Then $\{j_1\}$, $\{j_1\to j_2\},\ldots,\{j_1\to\cdots\to j_{|U|}\}$ are all $j$-mixing-sequences, and notice that $(\bm{\bar y},\bm{\bar z})$ satisfies the mixing inequalities corresponding to all these $j$-mixing-sequences at equality. In particular, it follows that $(\bm{\bar y}, \bm{\bar z})$ satisfies $z_{j_1}=z_{j_2}=\cdots=z_{j_{|U|}}$ at equality. There are only two points in $\{0,1\}^n$ that satisfy both of the constraints $z_i\leq 1$ for $i\notin U$ and $z_{j_1}=z_{j_2}=\cdots=z_{j_{|U|}}$ at equality; these points are $\mathbf{1}$ and %
	$\mathbf{1}_{[n]\setminus U}$. 
	Let $\bm{y^1},\bm{y^2}\in\R^k$ be such that $(\bm{y^1},\mathbf{1}),(\bm{y^2},%
	\mathbf{1}_{[n]\setminus U})
	\in\cM(\vW,\mathbf{0},\varepsilon)$. Then we have
	\[
	\sum_{j\in[k]}y^1_j\geq \varepsilon\quad \text{and}\quad \sum_{j\in[k]}y^2_j\geq \sum_{j\in[k]}\max\limits_{i\in U}\{w_{i,j}\}.
	\]
	As $\sum_{j\in[k]}\max\limits_{i\in U}\{w_{i,j}\}>\varepsilon$ by our assumption and $\sum_{j\in[k]}\bar y_j=\varepsilon$, $(\bm{\bar y},\bm{\bar z})$ cannot be a convex combination of $(\bm{y^1},\mathbf{1})$ and $(\bm{y^2},\mathbf{1}_{[n]\setminus U})$, implying in turn that $(\bm{\bar y},\bm{\bar z})$ does not belong to $\conv(\cM(\vW,\mathbf{0},\varepsilon))$.
	
	Now consider the case when~\eqref{condition1} is violated. Then there exist $p\in [n]\setminus \bar I(\varepsilon)$ and $q\in \bar I(\varepsilon)$ such that $w_{q,j}>w_{p,j}$ for some $j\in[k]$. In particular, $\sum_{j\in[k]}w_{p,j}<\sum_{j\in[k]}\max\{w_{p,j},w_{q,j}\}$. Let us consider the point $(\bm{\bar y}, \bm{\bar z})$ where
	\[
	\bar z_i=\begin{cases}
	\frac{1}{2}&\text{if}~i\in\{p,q\}\\
	1&\text{if}~i\notin \{p,q\}
	\end{cases}, \quad\text{and}
	\]
	\[
	\bar y_j=\begin{cases}
	\frac{1}{2}\max\left\{w_{p,j},~w_{q,j}\right\},&\text{if}~j\in[k-1]\\
	\frac{1}{2}\max\left\{w_{p,k},~w_{q,k}\right\}+\frac{1}{2}\left(\varepsilon+\sum\limits_{j\in[k]}w_{p,j}-\sum\limits_{j\in[k]}\max\left\{w_{p,j},~w_{q,j}\right\}\right),&\text{if}~j=k
	\end{cases}
	\]
	By definition of $\bm{\bar y}$, we always have $\sum_{j\in[k]}\bar y_j=\frac{1}{2}\left(\varepsilon+\sum\limits_{j\in[k]}w_{p,j}\right)>\varepsilon$, where the inequality follows from $p\notin \bar I(\epsilon)$. Moreover, as $p\in [n]\setminus \bar I(\varepsilon)$ and $q\in \bar I(\varepsilon)$, we have $\sum_{j\in[k]}w_{p,j}>\varepsilon\geq \sum_{j\in[k]}w_{q,j}$, and hence 
	\[
	\varepsilon+\sum\limits_{j\in[k]}w_{p,j}-\sum\limits_{j\in[k]}\max\left\{w_{p,j},~w_{q,j}\right\}\geq \sum_{j\in[k]}w_{q,j} +\sum\limits_{j\in[k]}w_{p,j}-\sum\limits_{j\in[k]}\max\left\{w_{p,j},~w_{q,j}\right\}=\sum\limits_{j\in[k]}\min\left\{w_{p,j},~w_{q,j}\right\}\geq0, 
	\]
	where the last inequality follows from the fact that $w_{i,j}\geq0$ for all $i\in[n]$ and $j\in[k]$.  
	So, it follows that
	\[
	\bar y_k\geq \frac{1}{2}\max\left\{w_{p,k},~w_{q,k}\right\}.
	\]
	As before, we can argue that $(\bm{\bar y},\bm{\bar z})$ satisfies the mixing inequalities. Now we argue that $(\bm{\bar y},\bm{\bar z})$ satisfies every aggregated mixing inequality. By Lemma~\ref{lem:aggregate-lemma}, it is sufficient to consider only the sequences $\Theta=\left\{i_1\to\cdots\to i_\theta\right\}$ that are contained in $\{p,q\}$.  Since $\Theta\subseteq \{p,q\}$, we know that $\bar z_{i_1}=\cdots=\bar z_{i_\theta}=\frac{1}{2}$. Then, the following inequality~\eqref{eq2-4} implies~\eqref{eq1}.  
	\begin{equation}
	\sum_{j\in[k]}\bar y_j=\frac{1}{2}\left(\varepsilon+\sum\limits_{j\in[k]}w_{p,j}\right)\geq\frac{1}{2}\min\left\{\varepsilon,~L_{\vW,\Theta}\right\} +\frac{1}{2}\sum_{j\in[k]} \max\left\{w_{i,j}:~i\in\Theta\right\}.\label{eq2-4}
	\end{equation}
	When $\Theta$ contains both $p$ and $q$, we have $L_{\vW,\Theta}=\sum_{j\in[k]}\min\{w_{p,j},w_{q,j}\}\leq \sum_{j\in[k]} w_{q,j}\leq \varepsilon$ (since $q\in\bar I(\epsilon)$) and $\sum_{j\in[k]} \max\left\{w_{i,j}:~i\in\Theta\right\}=\sum_{j\in[k]} \max\left\{w_{p,j},w_{q,j}\right\}$. Then the right-hand side of~\eqref{eq2-4} is 
	\[
	\frac{1}{2} \left(\sum_{j\in[k]}\min\{w_{p,j},w_{q,j}\}+\sum_{j\in[k]}\min\{w_{p,j},w_{q,j}\}\right)=\frac{1}{2}\sum_{j\in[k]}w_{p,j}+\frac{1}{2}\sum_{j\in[k]}w_{q,j},
	\]
	so inequality~\eqref{eq2-4} holds in this case since $q\in\bar I(\epsilon)$. If $\Theta=\{p\}$ or $\Theta=\{q\}$, inequality \eqref{eq2-4} clearly holds. Consequently, Lemma~\ref{lem:aggregate-lemma} implies that $(\bm{\bar y},\bm{\bar z})$ satisfies the aggregated mixing inequalities~\eqref{aggregate} for all sequences as well. Suppose for a contradiction that $(\bm{\bar y},\bm{\bar z})$ is a convex combination of two points $(\bm{y^1},\bm{z^1})$ and $(\bm{y^2},\bm{z^2})$ in $\cM(\vW,\mathbf{0},\varepsilon)$. As the previous case, we can argue that $\bm{z^1}$ and $\bm{z^2}$ satisfy  $z_p=z_q$ and $z_i\leq 1$ for $i\not\in \{p,q\}$  at equality, and therefore, $\bm{z^1}=\mathbf{1}$ and $\bm{z^2}=\mathbf{1}_{[n]\setminus \{p,q\}}$.
	Then we have
	\[
	\sum_{j\in[k]}y^1_j\geq \varepsilon,\quad\sum_{j\in[k]}y^2_j\geq \sum_{j\in[k]}\max\{w_{p,j},w_{q,j}\}\quad\text{and}\quad (\bm{\bar y},\bm{\bar z})=\frac{1}{2}(\bm{y^1},\bm{z^1})+\frac{1}{2}(\bm{y^2},\bm{z^2}),
	\]
	which implies that
	\[
	\frac{1}{2}\left(\varepsilon+\sum_{j\in[k]}w_{p,j}\right) = \sum_{j\in[k]} \bar y_j=\frac{1}{2} \sum_{j\in[k]} (y^1_j+y^2_j)\geq \frac{1}{2}\varepsilon+\frac{1}{2}\sum_{j\in[k]}\max\{w_{p,j},w_{q,j}\}.
	\]
	This is a contradiction, because we assumed $\sum_{j\in[k]}w_{p,j}<\sum_{j\in[k]}\max\{w_{p,j},w_{q,j}\}$. Therefore, $(\bm{\bar y},\bm{\bar z})$ is not contained in $\conv(\cM(\vW,\mathbf{0},\varepsilon))$, as required.
	
	In order to finish the proof we consider the case of $\varepsilon >L_{\vW}(\varepsilon)$. Based on the previous parts of the proof, we may assume that $\bar I(\varepsilon)$ is $\varepsilon$-negligible. Then, $L_{\vW}(\varepsilon)$ is finite, and thus there exist distinct $p,q\in [n]\setminus \bar I(\varepsilon)$ such that $\varepsilon>\sum_{j\in[k]}\min\left\{w_{p,j},~w_{q,j}\right\}=L_{\vW}(\varepsilon)$. Let us consider the point $(\bm{\bar y}, \bm{\bar z})$ where
	\[
	\bar z_i=\begin{cases}
	\frac{1}{2}&\text{if}~i\in\{p,q\}\\
	1&\text{if}~i\notin \{p,q\}
	\end{cases}
	\quad
	\text{and}
	\quad
	\bar y_j=\begin{cases}
	\frac{1}{2}\max\left\{w_{p,j},~w_{q,j}\right\}&\text{if}~j\in[k-1]\\
	\frac{1}{2}\max\left\{w_{p,k},~w_{q,k}\right\}+\frac{1}{2}L_{\vW}(\varepsilon)&\text{if}~j=k
	\end{cases}.
	\]
	Then 
	\[
	\sum_{j\in[k]}\bar y_j=\sum_{j\in[k]}\frac{1}{2}\max\left\{w_{p,j},~w_{q,j}\right\}+\frac{1}{2}\sum_{j\in[k]}\min\left\{w_{p,j},~w_{q,j}\right\}=\frac{1}{2}\sum_{j\in[k]}w_{p,j}+\frac{1}{2}\sum_{j\in[k]}w_{q,j}>\varepsilon,
	\]
	where the first equation follows from the properties of $L_{\vW}(\varepsilon)$ in this case, and the inequality follows from our assumption that $p,q\in [n]\setminus \bar I(\varepsilon)$. Similar to the previous cases, we can argue that $(\bm{\bar y},\bm{\bar z})$ satisfies the mixing inequalities. Now we argue that $(\bm{\bar y},\bm{\bar z})$ satisfies every aggregated mixing inequality. By Lemma~\ref{lem:aggregate-lemma}, it is sufficient to consider only the sequences $\Theta=\left\{i_1\to\cdots\to i_\theta\right\}$ contained in $\{p,q\}$.  Since $\Theta\subseteq \{p,q\}$, we know that $\bar z_{i_1}=\cdots=\bar z_{i_\theta}=\frac{1}{2}$. Then the following inequality~\eqref{eq2-3} implies~\eqref{eq1}.
	\begin{equation}
	\sum_{j\in[k]}\bar y_j=\frac{1}{2}\sum_{j\in[k]}w_{p,j}+\frac{1}{2}\sum_{j\in[k]}w_{q,j}\geq\frac{1}{2}\min\left\{\varepsilon,~L_{\vW,\Theta}\right\} +\frac{1}{2}\sum_{j\in[k]} \max\left\{w_{i,j}:~i\in\Theta\right\}.\label{eq2-3}
	\end{equation}
	When $\Theta$ contains both $p$ and $q$, we have 
	\[
	L_{\vW,\Theta}=\sum_{j\in[k]}\min\{w_{p,j},w_{q,j}\}\quad\text{and}\quad\sum_{j\in[k]} \max\left\{w_{i,j}:~i\in\Theta\right\}=\sum_{j\in[k]} \max\left\{w_{p,j},w_{q,j}\right\}.
	\] 
	Therefore,~\eqref{eq2-3} holds in this case. \eqref{eq2-3} clearly holds if $\Theta=\{p\}$ or $\Theta=\{q\}$, because $\varepsilon$ is smaller than $\sum_{j\in[k]}w_{p,j}$ and $\sum_{j\in[k]}w_{q,j}$ (this follows from $p,q\notin \bar I(\varepsilon)$). Consequently, Lemma~\ref{lem:aggregate-lemma} implies that $(\bm{\bar y},\bm{\bar z})$ satisfies the aggregated mixing inequalities~\eqref{aggregate} for all sequences as well. Suppose for a  contradiction that $(\bm{\bar y},\bm{\bar z})$ is a convex combination of two points $(\bm{y^1},\bm{z^1})$ and $(\bm{y^2},\bm{z^2})$ in $\cM(\vW,\mathbf{0},\varepsilon)$. As in the previous cases, we can argue that $\bm{z^1}$ and $\bm{z^2}$ satisfy the constraints $z_i\leq 1$ for $i\not\in \{p,q\}$ and $z_p=z_q$ at equality. Therefore, $\bm{z^1}=\mathbf{1}$ and $\bm{z^2}=\mathbf{1}_{[n]\setminus \{p,q\}}$.
	Then we have
	\[
	\sum_{j\in[k]}y^1_j\geq \varepsilon,\quad\sum_{j\in[k]}y^2_j\geq \sum_{j\in[k]}\max\{w_{p,j},w_{q,j}\}\quad\text{and}\quad (\bm{\bar y},\bm{\bar z})=\frac{1}{2}(\bm{y^1},\bm{z^1})+\frac{1}{2}(\bm{y^2},\bm{z^2}),
	\]
	which implies that
	\[
	\sum_{j\in[k]}\frac{1}{2}\max\left\{w_{p,j},~w_{q,j}\right\}+\frac{1}{2}\sum_{j\in[k]}\min\left\{w_{p,j},~w_{q,j}\right\} = \sum_{j\in[k]} \bar y_j=\frac{1}{2} \sum_{j\in[k]} (y^1_j+y^2_j)\geq \frac{1}{2}\varepsilon+\frac{1}{2}\sum_{j\in[k]}\max\{w_{p,j},w_{q,j}\}.
	\]
	This is a contradiction to our assumption that $\varepsilon>\sum_{j\in[k]}\min\left\{w_{p,j},~w_{q,j}\right\}$. Therefore, $(\bm{\bar y},\bm{\bar z})$ is not contained in $\conv(\cM(\vW,\mathbf{0},\varepsilon))$. This completes the proof of Theorem~\ref{thm:main}.

\end{document}